\documentclass[12pt]{amsart}

\usepackage{amsmath,amsfonts,amsthm,amsopn,cite,mathrsfs} 
\usepackage{epsfig,verbatim}
\usepackage{subfigure}
\usepackage{enumerate}
\usepackage[pdftex,%
plainpages=false,%
pdfpagelabels,%
hypertexnames=true,% true makes index work
colorlinks=true,%
linkcolor=black,%
citecolor=black,%
bookmarksopen=true,%
bookmarksopenlevel=3,%
bookmarksnumbered=true]{hyperref}

\allowdisplaybreaks

\newtheorem{thm}{Theorem}[section]
\newtheorem{prop}[thm]{Proposition}
\newtheorem{lem}[thm]{Lemma}
\newtheorem{cor}[thm]{Corollary}

\theoremstyle{definition}
\newtheorem{defi}[thm]{Definition}
\newtheorem{rem}[thm]{Remark}

\renewcommand{\epsilon}{\varepsilon}
\renewcommand{\phi}{\varphi}

\newcommand{\ie}{i.e.\@,\@ }

\newcommand{\etal}{et al.\@}

\newcommand{\R}{\mathbb{R}}
\newcommand{\Rd}{\mathbb{R}^d}
\newcommand{\Rtd}{\mathbb{R}^{2d}}
\newcommand{\Z}{\mathbb{Z}}

\newcommand{\T}{\mathbb{T}}

\newcommand{\Qf}{\mathbb{Q}}

\newcommand{\G}{\mathcal{G}}

\newcommand{\F}{\mathcal{F}}
\newcommand{\cK}{\mathcal{K}}
\newcommand{\U}{\mathcal{U}}
\newcommand{\D}{\mathcal{D}}

\newcommand{\Sc}{\mathcal{S}}

\newcommand{\I}{\mathcal{I}}
\newcommand{\A}{\mathcal{A}}
\newcommand{\Q}{\mathcal{Q}}

\newcommand{\RS}{R}

\newcommand{\stft}{short-time Fourier transform}

\newcommand{\tf}{time-frequency}

\newcommand{\fif}{if and only if}
\newcommand{\tfs}{time-frequency shift}
\newcommand{\modsp}{modulation space}

\newcommand*\dif{\mathop{}\!\mathrm{d}}

\makeatletter
\DeclareRobustCommand\widecheck[1]{{\mathpalette\@widecheck{#1}}}
\def\@widecheck#1#2{%
    \setbox\z@\hbox{\m@th$#1#2$}%
    \setbox\tw@\hbox{\m@th$#1%
       \widehat{%
          \vrule\@width\z@\@height\ht\z@
          \vrule\@height\z@\@width\wd\z@}$}%
    \dp\tw@-\ht\z@
    \@tempdima\ht\z@ \advance\@tempdima2\ht\tw@ \divide\@tempdima\thr@@
    \setbox\tw@\hbox{%
       \raise\@tempdima\hbox{\scalebox{1}[-1]{\lower\@tempdima\box
\tw@}}}%
    {\ooalign{\box\tw@ \cr \box\z@}}}
\makeatother

\newcommand{\Fhat}{\widehat{\phantom{x}}}  %% without negative space

\newcommand{\Rimp}[2]{(#1)\,$\Rightarrow$(#2)}

\newcommand{\LRimp}[2]{(#1)\,$\Leftrightarrow$(#2)}

\newcommand{\duall}{\Lambda{ ^\circ}}
\def\inv{^{-1}}
\def\rd{\R ^d}

\def\zd{\Z ^d}
\def\td{\T ^d}
\def\rdd{{\R ^{2d}}}
\def\zdd{{\Z ^{2d}}}

\def\lrd{L^2(\rd)}
\def\lrdd{L^2(\rdd)}
\def\intrd{\int_{\rd}}
\def\intrdd{\int_{\rdd}}

\newcommand{\lnrm}{_{\ell^2}}
\newcommand{\Lnrm}{_{L^2}}

\DeclareMathOperator{\vspan}{span}
\DeclareMathOperator{\vol}{vol}

\DeclareMathOperator*{\einf}{ess\,inf}

\DeclareMathOperator{\supp}{supp}

\begin{document}
\begin{abstract}
This chapter offers a systematic and streamlined exposition of the
most important  characterizations of Gabor frames over a lattice. 
\end{abstract}

\title{Gabor Frames: Characterizations  and  Coarse Structure}
\author{Karlheinz Gr\"ochenig}
\address{Faculty of Mathematics \\
University of Vienna \\
Oskar-Morgenstern-Platz 1 \\
A-1090 Vienna, Austria}
\email{karlheinz.groechenig@univie.ac.at}
\author{Sarah Koppensteiner}
% \address{Faculty of Mathematics \\
% University of Vienna \\
% Oskar-Morgenstern-Platz 1 \\
% A-1090 Vienna, Austria}
\email{sarah.koppensteiner@univie.ac.at}

\subjclass[2000]{}
\date{}
% \thanks{K.\ G.\ was
%   supported in part by the  project P26273 - N25  of the
% Austrian Science Fund (FWF)}
\maketitle

\section{Introduction}

Given a point $z= (x,\xi ) \in \rdd $ in time-frequency space (phase
space), we define the corresponding time-frequency shift  $\pi (z)$
acting on a function $f\in \lrd $ by
$$
\pi (z) f(t) = e^{2\pi i \xi \cdot t} f(t-x) \, .
$$
Gabor analysis deals with the spanning properties of sets of \tfs
s. Specifically, for a window function $g\in \lrd $ and a discrete set
$\Lambda \subseteq \rdd $, which we will always assume to be a
lattice, we would like to understand when the set
$$
\G (g,\Lambda ) = \{ \pi (\lambda )g: \lambda \in \Lambda \}
$$
is a frame. This means that there exist positive constants $A,B>0$ such that
\begin{equation} \label{eq:fr1}
    A \|f\|\Lnrm^2 
    \leq \sum_{\lambda \in \Lambda} | \langle f, \pi(\lambda) g \rangle |^2 
    \leq B \|f\|\Lnrm^2 \qquad \forall f \in L^2(\Rd) \text{.}
  \end{equation}
For historical reasons a frame with this structure is called a
\emph{Gabor frame}, or sometimes a Weyl-Heisenberg frame.

The motivation for studying sets of \tfs s is in the foundations of
quantum mechanics by J.\ von Neumann~\cite{neumann} and in information
theory by D.\ Gabor~\cite{ga46}. Since 1980 the investigation of
Gabor frames has stimulated the interest of many mathematicians in
harmonic,  complex, and numerical  analysis and engineers in signal
processing and wireless communications.

Whereas \eqref{eq:fr1} expresses a strong form of completeness (with
stability built in the definition), a complementary concept is the
linear independence  of \tfs s. Specifically, we ask for constants
$A,B>0$ such that
\begin{equation}
  \label{eq:c8}
  A \|c\|\lnrm^2 \leq \Big\| \sum _{ \lambda \in \Lambda  } c_\lambda \pi
  (\lambda ) g \Big\|\Lnrm^2 \leq B \|c\|\lnrm^2  \qquad   \forall c\in
  \ell^2(\Lambda ) \, ,
\end{equation}
and  in this case $\G (g,\Lambda )$ is called a Riesz sequence in $\lrd $. Riesz
sequences are important  in wireless communications: a data set $(c_\lambda
)_{\lambda \in \Lambda }$ is transformed into an analog signal $f=
\sum _{\lambda \in \Lambda } c_\lambda \pi (\lambda )g$ and then
transmitted. The task at the receiver is to decode the data
$(c_\lambda )$. In this context \eqref{eq:c8} expresses the fact that
the coefficients $c_\lambda $ are uniquely determined by $f$ and that
their recovery is feasible in a robust way.

In this chapter we restrict our attention to sets of \tfs s over a
lattice $\Lambda $, i.e., $\Lambda = A \zdd $ for an invertible,
real-valued $2d\times 2d$ matrix $A$. The lattice structure implies
the translation invariance $\pi (\lambda ) \G (g,\Lambda ) = \G
(g,\Lambda ) $ (up to
phase factors) and is at the basis of a  beautiful and deep structure
theory and many
characterizations of \eqref{eq:fr1} and \eqref{eq:c8}. 

After three decades we have a clear understanding of the structures
governing Gabor systems.  Our goal is to  collect the most important
characterizations of Gabor frames and  offer a systematic exposition
of these structures. 
In the center of these characterizations is the duality theorem for
Gabor frames. To our knowledge \emph{all} other characterizations
within the $L^2$-theory follow directly from this fundamental
duality. In particular,  the celebrated characterizations of Janssen and
Ron-Shen are  consequences of  the duality theorem, and the
characterization of Zeevi and Zibulski for rational lattices also
becomes a corollary.

Even with this impressive list of different criteria at our disposal,
it remains very difficult to determine whether a given window function
and lattice generate a Gabor frame.  Ultimately, each criterion
(within the $L^2$-theory) is formulated by means of the invertibility
of some operator, and  % The proof of these characterizations amounts to
% showing the unitary equivalence of these operators, which
% is
% completely different and usually easier than
proving 
invertibility is always difficult.  This fact explains perhaps why there are so many general
results about Gabor frames, but so few explicit results about concrete
Gabor frames.

Yet, there are some success stories due to Lyubarski~\cite{MR1188007},
Seip~\cite{MR1173117}, Janssen~\cite{MR1955931,MR1964306}, and some
recent progress for totally positive
windows~\cite{MR3053565,grrost17}. All these results have applied some
of the characterizations presented here, or even invented some new
ones.
% characterization of Gabor frames.
On the other hand, most questions
about concrete Gabor systems remain unanswered, and so far every
explicit conjecture about Gabor frames (with one exceptation) has been
disproved by counter-examples.

To document some of the many white spots on the  map of Gabor frames,
let us mention two specific examples. (i) Let $g_1(t) = (1-|t|)_+$ be the
hat function (or $B$-spline of order $1$). It is known that for all $\alpha >0$ the Gabor
system $\G (g_1, \alpha \Z \times 2\Z )$ is \emph{not} a Gabor
frame.  But it is not known whether   $\G (g_1,
0.33 \Z \times 2.001 \Z )$ is a frame. (ii) Let $h_1 (t) = t
e^{-\pi t^2}$ be the first Hermite function. It is known that  $\G
(h_1, \alpha \Z \times \beta\Z )$ is \emph{not} a Gabor 
frame whenever  $\alpha \beta = 2/3$~\cite{MR3027914}. But it is not known
whether   $\G (h_1,
\Z \times 0.66666 \Z )$ is a frame. In both cases, there is
numerical evidence that these Gabor systems are frames, but so far
there is no proof despite an abundance of precise criteria to check.

The novelty of our approach  is the streamlined sequence of proofs, so
that most of the structure theory of Gabor
frames fits  into a single, short chapter. In view of dozens of
efforts on every aspect of Gabor analysis, we hope that this survey
will be useful and inspire work on concrete open  questions.   The only prerequisite is
the thorough mastery of the Poisson summation formula and some basic
facts about frames and Riesz sequences. 

The chapter is organized as follows: Section~2 covers the main objects
of Gabor analysis. Section~3 is devoted to the interplay between the
\stft ,  the Poisson summation formula, and commutativity of \tfs
s. The central Section~4 offers a complete proof of the duality
theorem for Gabor frames. Section~5 sketches the main theorems about
the coarse structure of Gabor frames. A list of criteria that are
tailored to rectangular frames is discussed and proved in
Section~6. In Section~7 we derive the criterium of Zeevi and Zibulski
for rational lattices, and Section~8 presents a number of (technically
more advanced) criteria some of which have recently become
useful. Except for the last section, we fully prove all statements.

\section{The Objects of Gabor Analysis}

Let $g \in L^2(\Rd)$ be a non-zero window function and
$\Lambda \subseteq \Rtd$ a lattice.  The set
$\G(g, \Lambda)$ is called a \emph{Gabor frame} if there exist
positive constants $A,B > 0$ such that
\begin{equation}
  \label{eq:frame-ineq}
  A \|f\|\Lnrm^2 
  \leq \sum_{\lambda \in \Lambda} | \langle f, \pi(\lambda) g \rangle |^2 
  \leq B \|f\|\Lnrm^2 \qquad \forall f \in L^2(\Rd) \text{.}
\end{equation}

The frame inequality \eqref{eq:frame-ineq} can be recast by means of
functional analytic properties  of certain operators associated to a
Gabor system $\G (g,\Lambda )$. We will use the \index{frame operator} \index{Gabor frame operator}
\emph{frame operator} $S= S_{g,\Lambda}$ defined by 
\begin{equation*}
  S_{g,\Lambda}f
  = \sum_{\lambda \in \Lambda} \langle f, \pi (\lambda) g  \rangle
  \pi (\lambda) g \text{.}
\end{equation*}
Then $\G (g,\Lambda )$ is % a Bessel sequence, \fif\ $S_{g,\Lambda }$ is
% bounded, and 
a frame \fif\ $S_{g,\Lambda }$ is bounded and invertible on $L^2(\rd
)$. The extremal spectral values $A,B$ are called the frame bounds. If
they can be chosen to be equal $A=B$, then the frame operator is a
multiple of the identity, and $\G (g,\Lambda )$ is called a tight
frame. 

We will also use the  \index{Gramian operator} \emph{Gramian operator}
$ G= G_{g, \Lambda}$ defined by its entries
\begin{equation*}
  G_{\lambda \mu} = \langle \pi(\mu) g, \pi(\lambda)g \rangle \text{.}
\end{equation*}
In this notation, $\G (g,\Lambda )$ is a Riesz sequence \fif\
$G_{g,\Lambda }$ is bounded and invertible on $\ell ^2(\Lambda )$. 

If the upper inequality in \eqref{eq:frame-ineq} is satisfied, then
the frame operator is well-defined and  bounded on $L^2(\Rd)$ and the
Gramian operator is bounded on 
$\ell^2(\Lambda)$.  In this case, we call
$\G(g, \Lambda)$ a \emph{Bessel sequence}.

  The underlying
object of this definition is the \emph{\stft} of $f$ with respect to the
window function $g\in \lrd $, which is defined by
\begin{equation*}
  \label{eq:c10}
  V_gf (z) = V_gf(x,\xi ) = \intrd f(t) \bar{g}(t-x) e^{-2\pi i \xi
    \cdot t} \, \dif t \, .
\end{equation*}
We will need the following properties of the \stft .

\begin{lem}[Covariance property]
  \label{sec:short-time-fourier-CovarianceProp}
  Let $f,g\in L^2(\Rd)$ and $w, z\in \rdd $. Then
  \begin{equation}
    \label{eq:2}
    V_g{(\pi (w) f)}(z) = e^{-2\pi i (z_2 - w_2) \cdot w_1}
    V_{g}{f}(z-w) \qquad \text{ and }
  \end{equation}  
  \begin{equation}
    \label{eq:3}
    V_{\pi(w) g}(\pi (w)f)(z) = e^{2\pi i z \cdot \mathcal{I} w }
    V_{g}f (z) \text{,}
  \end{equation}
  % with $\I = \left(\begin{smallmatrix}
  %     0 &  I_d \\
  %     -I_d & 0
  %   \end{smallmatrix}\right)$
where $\I =
\left(\begin{smallmatrix}
  0 &  I_d \\
  -I_d & 0
\end{smallmatrix}\right)$ denotes the standard symplectic matrix and
$I_d$ is the $d$-di\-men\-sion\-al identity matrix.
\end{lem}

\begin{prop}[Orthogonality relations]
  \label{sec:short-time-fourier-OrthoRel}
  Let $ f, g, h , \gamma \in L^2(\Rd)$.
  \begin{enumerate}[(i)]
  \item Then $V_g f, V_\gamma h \in L^2(\Rtd)$ and 
    \begin{equation}
      \label{eq:short-time-fourier-OrthoRel}
      \langle V_g f, V_\gamma h \rangle_{L^2(\Rtd)} = \langle f, h \rangle_{L^2(\Rd)}
      \overline{\langle g, \gamma \rangle}_{L^2(\Rd)} \text{.}
    \end{equation}
    In particular, if $\|h\|\Lnrm = 1$, then $V_h$ is an isometry from
    $\lrd $ to $\lrdd $.
  \item Furthermore, for all $z \in \Rtd$,
    \begin{equation}
      \label{eq:modulation-spaces-eq4}
      \left( V_{g}{f} \cdot \overline{V_{\gamma}{h}} \right)\! \Fhat (z) 
      = \left(V_{g}{\gamma } \cdot \overline{V_ hf} \right)\!(\mathcal{I} z) \text{.}
    \end{equation}
  \end{enumerate}
\end{prop}

\begin{proof}
  (i) The orthogonality
  relations~\eqref{eq:short-time-fourier-OrthoRel} are a well
  established fact from representation theory. For a direct proof
  using only Plancherel's theorem we refer to the
  textbooks~\cite{daubechies92,MR1843717}.

  (ii) By the Cauchy-Schwarz inequality the product
  $V_{g}{f} \cdot \overline{V_{\gamma}{h}}$ is in $L^1(\Rtd)$,
  therefore the Fourier transform is defined pointwise, and we obtain
  \begin{align*}
    \left( V_{g}{f} \cdot \overline{V_{\gamma}{h}} \right)\!\Fhat (z) &=
    \int_{\Rtd} V_{g}{f}(w) \cdot \overline{V_{\gamma}{h}(w)}
    e^{-2 \pi i w \cdot z}\dif w\\
    & = \int_{\Rtd} V_{\pi
        (\mathcal{I}z)g}\big(\pi
        (\mathcal{I}z)f\big)(w) \cdot \overline{V_{\gamma}
      h(w)}\dif w \\
    & = \langle \gamma  ,
      \pi (\mathcal{I}z) g \rangle  \overline{\langle h,
      \pi (\mathcal{I}z)f \rangle } \text{,}
  \end{align*}
  where we first used $\mathcal{I}^2 = - I_{2d}$ and the covariance
  property \eqref{eq:3}, then the orthogonality
  relations~\eqref{eq:short-time-fourier-OrthoRel} to separate the
  integral into two inner products.
\end{proof}

\section{Commutation Rules and the Poisson Summation Formula in Gabor 
  Analysis}

In this section we exploit the  invariance properties of a Gabor system
for  the structural interplay between the 
\stft\ and \tf\ lattices.

\subsection{Poisson Summation Formula}
If $\Lambda $ is a lattice, then the function
$\Phi (z) = \sum _{\lambda \in \Lambda } |\langle f, \pi (z+\lambda )g
\rangle |^2$
satisfies $\Phi (z+\nu ) = \Phi (z)$ for $\nu \in \Lambda $ and thus
is periodic with respect to $\Lambda $. It is therefore natural to
study the Fourier series of $\Phi $.  The mathematical tool is the
Poisson summation formula, and this is in fact the mathematical core
of all existing characterizations of Gabor frames over a lattice
(though the terminology is often a bit different, e.g., fiberization
technique in~\cite{MR1460623}).

We formulate the Poisson summation formula explicitly for an arbitrary
lattice $\Lambda = A \zdd $ where $A$ denotes an invertible, real-valued
$2d \times 2d$ matrix.  We write $\Lambda ^\perp = (A^T)\inv \zdd $
for the dual lattice and $\Lambda ^\circ = \mathcal{I} \Lambda ^\perp$
for the adjoint lattice with % the standard symplectic matrix 
$\I =
\left(\begin{smallmatrix}
  0 &  I_d \\
  -I_d & 0
\end{smallmatrix}\right)$.

The \emph{volume} of the lattice is 
$\vol{\Lambda} = |\det(A)|$, and   the reciprocal value
$D(\Lambda)=\vol(\Lambda)^{-1}$ is  the \emph{density} or
\emph{redundancy} of $\Lambda $.

We first formulate a sufficiently general version of the Poisson
summation formula~\cite{MR0304972}. 

\begin{lem}
  Assume that $\Lambda = A \zdd $ and $F\in L^1(\rdd )$. Then the
  periodization
  $\Phi (x) = \sum _{\lambda \in \Lambda } F(x-\lambda )$ is in
  $L^1(\rdd / \Lambda )$.

  \begin{enumerate}[(i)]
  \item The Fourier coefficients of $\Phi $ are given by
    $\hat{\Phi }(\nu ) = \hat{F}(\nu )$ for all
    $\nu \in \Lambda ^\perp $.

  \item Poisson summation formula -- general version: Let $K_n$ be a
    summability kernel~\footnote{It suffices to take the Fourier
      coefficients of the multivariate Fejer kernel
      $\hat{F}_n(k) = \prod _{j=1}^d \big(1 - \frac{|k_j|}{n+1} \big)$
      and set $K_n (\nu ) = \hat{F}_n(A^T \nu) = \hat{F}_n(k)$ for
      $\nu = (A^T)\inv k \in \Lambda ^\perp $ }, then
    \begin{equation*}
      \label{eq:c1}
      \sum_{\lambda \in \Lambda} F(z+\lambda)
      = \vol(\Lambda)^{-1}\lim _{n\to \infty } \sum_{\nu \in \Lambda^{\perp}}  K_n(\nu )\hat{F}(\nu)  \,
      e^{2\pi i \nu \cdot z} \, .
    \end{equation*}
    with convergence in $L^1(\rdd / \Lambda )$.

  \item  If
    $(\hat{F}(\nu ))_{\nu \in \Lambda ^\perp } \in \ell ^1(\Lambda
    ^\perp )$,
    then the Fourier series converges absolutely and $\Phi $ coincides
    almost everywhere with a continuous function. 
  \end{enumerate}
\end{lem}

By applying the Poisson summation formula to the function $V_gf \cdot
\overline{V_\gamma h}$ and a lattice $\Lambda $, we obtain an
important identity for the analysis of  Gabor frames. %  %  which is
% % indeed called the ``Fundamental Identity of Gabor Analysis'' by
% % Feichtinger and Luef~\cite{MR2264211}. \
% In fact, it is so central that Feichtinger and
% Luef~\cite{MR2264211} called it \emph{the} ``Fundamental Identity of Gabor
% Analysis''. 
This technique is so ubiquitous in Gabor analysis, that
Janssen~\cite{MR1601115} and later Feichtinger and
Luef~\cite{MR2264211} called it \emph{the} ``Fundamental Identity of
Gabor Analysis''.

% The centrality of this technique was already evident to
% Janssen~\cite{MR1350700} and a decade later it is deservedly called the
% ``Fundamental Identity of Gabor Analysis''.

\begin{thm}
  \label{sec:fund-ident-gabor-FIGA}
  Let $f,g, h, \gamma \in L^2(\Rd)$, and $\Lambda = A
  \Z^{2d}$ be a lattice.
  \begin{enumerate}[(i)]
  \item Then
    \begin{equation}
      \label{eq:FIGA}
       \sum_{\lambda \in \Lambda} V_{g}f(z+\lambda) \overline{V_{\gamma}h(z+\lambda)}
      = \vol(\Lambda)^{-1} \lim _{n\to \infty } \sum_{\mu \in
        \Lambda^\circ} K_n (-\I \mu ) V_{g}\gamma(\mu)
      \overline{V_{f}h(\mu)} e^{2\pi i \mu \cdot \I z }
    \end{equation}
    with convergence in $L^1(\rdd / \Lambda )$. 
  
  \item Assume in addition that both $\G (g,\Lambda )$ and
    $\G (\gamma ,\Lambda )$ are Bessel sequences and that
    $\sum _{\mu \in \duall} |V_g\gamma (\mu ) | < \infty $. Then 
    \begin{equation}
      \label{eq:FIGAptw}
      \sum_{\lambda \in \Lambda} V_{g}f(z+\lambda) \overline{V_{\gamma}h(z+\lambda)}
      = \vol(\Lambda)^{-1} \sum_{\mu \in
        \Lambda^\circ}  V_{g}\gamma(\mu)
      \overline{V_{f}h(\mu)} e^{2\pi i \mu \cdot \I z} \quad \forall z \in \Rtd \text{.}
    \end{equation}
  \end{enumerate}
\end{thm}

\begin{proof}
  (i) We apply the Poisson summation formula to the product
  $V_gf \cdot \overline{V_\gamma h}$ and the lattice $\Lambda $ and
  obtain
  \begin{align*}
    \sum_{\lambda \in \Lambda} V_{g}f(z+\lambda) \overline{V_{\gamma}h(z+\lambda)}
    & = \vol(\Lambda)^{-1} \lim _{n\to \infty }  \sum_{\nu \in
      \Lambda^\perp} K_n(\nu ) 
      \left(V_{g}f \cdot \overline{V_{\gamma}h} \right)\! \Fhat (\nu)
      \, e^{2\pi i \nu \cdot z} \\
    & = \vol(\Lambda)^{-1} \lim _{n\to \infty } \sum_{\nu \in
      \Lambda^\perp} K_n(\nu ) 
      \left(V_{g}{\gamma} \cdot \overline{V_{f}{h}}\right)\!(\I \nu)  \, e^{2\pi i \nu \cdot z} \\
    & = \vol(\Lambda)^{-1} \lim _{n\to\infty } \sum_{\mu \in
      \Lambda^\circ} K_n(-\I \mu ) 
      V_{g}{\gamma}(\mu) \overline{V_{f}{h}(\mu)} \,  \, e^{2\pi i
      \mu \cdot \I z} \, ,
  \end{align*}
  where we used 
  Proposition~\ref{sec:short-time-fourier-OrthoRel} to rewrite the
  Fourier transform in the first line.

  (ii) If $\sum _{\mu \in \duall} |V_g\gamma (\mu)| <\infty $, then
  the right-hand side of \eqref{eq:FIGAptw} converges absolutely to a
  continuous function, and we do not need the summability kernel.
  Next we rewrite the left-hand side with the help of
  identity~\eqref{eq:2} as
  $$
  \Phi (z) = \sum _{\lambda \in \Lambda } V_g(\pi (-z)f)(\lambda )
  \overline{V_\gamma(\pi (-z)h)(\lambda ) }\, ,
  $$
  where, as so often, the phase factors cancel. Since $\G(g,\Lambda )$
  is a Bessel sequence with Bessel bound $B_g$, we know that
  \begin{equation*}
    \|V_g (\pi (-z)f - f)|_\Lambda \|\lnrm \leq B_g^{1/2} \|\pi (-z)f -
    f \|\Lnrm \text{.}
  \end{equation*}
  This means that the map $z\mapsto V_g(\pi (-z)f)|_\Lambda $ is
  continuous from $\R $ to $\ell ^2(\Lambda )$ for all $f\in \lrd $.
  Likewise, the map $z \mapsto V_\gamma(\pi (-z)h)|_\Lambda $ is 
  continuous for all $h \in \lrd$.

  This observation implies that the left-hand side is also a
  continuous function. Thus both sides of \eqref{eq:FIGAptw}  are continuous and
  coincide almost everywhere, therefore % Since two continuous functions that are equal
  % almost everywhere must be equal everywhere, we have proved the
  % pointwise equality of 
  \eqref{eq:FIGAptw} must hold everywhere. 
  % After using
  % $V_gf(x,\xi) = e^{-2\pi i x \cdot \xi}\,\overline{V_fg(-x,-\xi)}$
  % on both factors on the right-hand side, we obtain the desired
  % result.
\end{proof}

\subsection{Commutation Rules}
In the fundamental identity \eqref{eq:FIGA} the adjoint lattice $\duall $
appears as a consequence of the Poisson summation formula. We now
present a more structural property of the adjoint lattice. 

% A straight forward computation reveals that
% the adjoint lattice is characterized by a  commutation relation.
\begin{lem}
\label{sec:commutation-rules-cadjoint}
Let $\Lambda \subseteq \rdd $ be a lattice. Then its adjoint lattice
is characterized by the property 
\begin{equation*}
  \label{eq:FIGA-adjoint-Char}
  \Lambda^\circ = \{ \mu \in \Rtd : \pi(\lambda) \pi(\mu)
  = \pi(\mu) \pi(\lambda) \qquad \forall \lambda \in \Lambda \} \text{.}
\end{equation*}
\end{lem}

\begin{proof}
  Let $z\in \rdd $ and $\lambda = Ak \in \Lambda $ for some
  $k\in \zdd $. A straight forward computation yields
  $\pi (z ) \pi (\lambda) = e^{2\pi i (\lambda _1 \cdot z_2 - \lambda
    _2 \cdot z_1)} \pi (\lambda ) \pi (z)$.
  Consequently, the time-frequency shifts commute
  \fif\
  \begin{equation*}
    1 = e^{2\pi i (\lambda _1 \cdot z_2 - \lambda _2 \cdot z_1)} =
    e^{2\pi i \lambda \cdot \I z} = e^{2\pi i Ak \cdot \I z} \, .
  \end{equation*}
  This holds for all $k\in \zdd $ \fif\
  $Ak \cdot \I z = k \cdot A^T \I z \in \Z$ for all $k\in \zdd $,
  which is precisely the case when $A^T \I z \in \zdd $, or equivalently when
  $z\in \I \inv (A^T)\inv \zdd = \duall $.
\end{proof}

  This interpretation of the adjoint lattice is crucial for an
  important technical point.  
\begin{lem}[Bessel duality] \label{cbessel}
Let $g\in \lrd $ and $\Lambda \subseteq \rdd $ be a lattice. Then 
$\G (g,\Lambda ) $ is a Bessel sequence if and only if $\G (g,
\Lambda ^\circ )$ is a Bessel sequence. 
\end{lem}

\begin{proof}
  The proof is inspired by \cite{MR1350701}.

  Fix  $h\in \Sc (\rd )$ with $\|h\|\Lnrm = 1$. Then $\G (h, M)$ is a
  Bessel sequence for every lattice $ M \subseteq \rdd $. Next let
  $Q=A[0,1)^{2d}$ be a fundamental domain of $\Lambda = A \zdd$, i.e.,
  $\rdd = \bigcup _{\lambda \in \Lambda } \lambda +Q $ as a disjoint
  union. As a consequence we may write $\int _{\rdd } f(z) \, \dif z =
  \int _Q \sum _{\lambda \in \Lambda } f(z-\lambda ) \, \dif z $ for $f\in
  L^1(\rdd )$. 

  Now assume that $\G (g,\Lambda )$ is a Bessel sequence. Let
  $c = (c_\mu )_{\mu \in \duall} \in \ell ^2(\duall )$ be a finite
  sequence and $f = \sum _{\mu \in \duall } c_\mu \pi (\mu )g$. Since
  $V_h: \lrd \to \lrdd $ is an isometry by
  Proposition~\ref{sec:short-time-fourier-OrthoRel}, we obtain
  \begin{equation}
    \begin{aligned}
      \label{eq:c4}
      \|f\|_{L^2(\rd )}^2 &= \|V_hf\|_{L^2(\rdd )} ^2 \\
      &= \int _{Q } \sum _{\lambda \in \Lambda } \Big| \langle \sum
      _{\mu \in \duall } c_\mu \pi (\mu )g, \pi (-\lambda +z)h \rangle
      \Big|^2 \, \dif z = \int _Q I(z) \, \dif z \, .
    \end{aligned}
  \end{equation}
  % \begin{equation}
  %   \label{eq:c4}
  %   \|f\|_{L^2(\rd )}^2 = \|V_hf\|_{L^2(\rdd )} ^2 = \int _{Q } \sum
  %   _{\lambda \in \Lambda } \Big| \langle \sum _{\mu \in \duall } c_\mu \pi (\mu
  %   )g, \pi (-\lambda +z)h \rangle \Big|^2 \, dz = \int _Q I(z) \, dz \, .
  % \end{equation}

  We now reorganize the sum over $\mu $.  First we use
  $\pi (-\lambda +z) = \gamma _{z,\lambda } \pi (\lambda )^* \pi (z)$
  for some phase factor $|\gamma _{z,\lambda }|=1$. Then we use the
  commutativity
  $\pi (\lambda ) \pi (\mu ) = \pi (\mu ) \pi (\lambda ) $ for all
  $\lambda \in \Lambda , \mu \in \duall $
  (Lemma~\ref{sec:commutation-rules-cadjoint}). This is the heart of
  the proof, and the reader should convince herself that the proof
  does not work without this property.  We obtain
  \begin{align*}
    \langle   \sum _{\mu \in \duall } c_\mu \pi (\mu
    )g, \pi (-\lambda +z)h \rangle
    &=  \bar{\gamma } _{\lambda , z}
    \langle \pi (\lambda ) g, \sum _{\mu \in \duall } \bar{c}_\mu \pi
    (\mu )^* \pi (z) h\rangle \\
    &=   \bar{\gamma } _{\lambda , z}
    \langle \pi (\lambda ) g, \pi (z) \sum _{\mu \in \duall }
    \bar{c}_{-\mu } \gamma _{\mu ,z} \pi(\mu )  h\rangle  \, .
  \end{align*}
  Since both Gabor families $\G (g,\Lambda )$ and $\G (h,\duall )$ are
  Bessel sequences by assumption (with constants $B_g $ and $B_h$), we
  obtain a pointwise estimate for the integrand $I(z)$ in
  \eqref{eq:c4}:
  \begin{align*}
    I(z) &= \sum _{\lambda \in \Lambda } \Big| \bar{\gamma } _{\lambda , z}
           \langle \pi (z) \sum _{\mu \in \duall }
           \bar{c}_{-\mu } \gamma _{\mu ,z} \pi
           (\mu )  h , \pi (\lambda ) g \rangle \Big|  ^2 \\
         &\leq B_g \Big\| \pi (z) \sum _{\mu \in \duall }
           \bar{c}_{-\mu } \gamma _{\mu ,z} \pi
           (\mu )  h \Big\|_{L^2(\rd) } ^2 \\
         &\leq B_g B_h \sum _{\mu \in \duall } |\bar{c}_{-\mu } \gamma _{\mu
           ,z} |^2 = B_g B_h \|c\|\lnrm^2 \, .
  \end{align*}
  Integration over $z$ now yields
  \begin{equation*}
    \Big\|\sum _{\mu \in \duall } c_\mu \pi (\mu )g \Big\|\Lnrm^2 = \int _Q I(z) \, \dif z
    \leq B_g B_h |\det A| \|c\|\lnrm^2 \, ,
  \end{equation*}
  and thus $\G (g, \duall )$ is a Bessel sequence.

  Since $\Lambda = (\duall ) ^\circ $, the proof of the converse is
  the same.
\end{proof}

\section{Duality Theory}
\label{sec:duality-theory}

% While the characterizations for Gabor frames in
% Corollary~\ref{sec:gabor-frames-FrameChar} were just functional
% analytic reinterpretations of the frame inequality, 
The duality theory relates the spanning properties of a Gabor family 
$\G(g, \Lambda)$ on a lattice $\Lambda $ to the spanning properties of
$\G(g, \Lambda^\circ)$ over the adjoint lattice. %  are
% intimately related: A Gabor frame on one side corresponds to a Riesz
% sequence on the other.  Furthermore, dual windows of $\G(g, \Lambda)$
% are characterized by a biorthogonality condition along the adjoint
% lattice. 
The following duality theorem is the central result of the theory of
Gabor frames. We will see that most structural results about Gabor frames can be
derived easily from it. 
\begin{thm}[Duality theorem]
  \label{sec:duality-theory-DualityTheory}
  Let $g \in L^2(\Rd)$ and $\Lambda \subseteq \Rtd$ be a lattice. Then
  the following are equivalent:
  \begin{enumerate}[(i)]
  \item $\G(g, \Lambda)$ is a frame for $L^2(\Rd)$.
  \item $\G(g, \Lambda^\circ)$ is a Bessel sequence and there exists a
    dual window $\gamma \in L^2(\Rd)$ such that $\G(\gamma,
    \Lambda^\circ)$ is a Bessel sequence satisfying
    \begin{equation}
      \label{eq:duality-theory-BiorthogonalityRelation}
      \langle \gamma , \pi(\mu)g \rangle
      = \vol(\Lambda) \delta_{\mu,0} \qquad \forall \mu \in \Lambda^\circ
      \text{.}
    \end{equation}
  \item $\G(g, \Lambda^\circ)$ is a Riesz sequence for $L^2(\Rd)$.
  \end{enumerate}
\end{thm}

% Condition \textit{(ii)} was discovered by the engineers Wexler and
% Raz~\cite{Wexler:1990:DGE:104676.104677}.  It is sometimes referred to
% as the Wexler-Raz biorthogonality relation and characterizes all
% possible dual windows (see
% Corollary~\ref{sec:duality-theorem-dual-CharDualWindows}).

We follow the proof sketch given in the survey
article~\cite{MR3232589}.
\begin{proof}
  %==========================================================%
  %                   BEGIN (i) ==> (ii)                     %

  \Rimp{i}{ii}: Since $\G(g, \Lambda)$ is a Gabor frame,
  there exists a dual window $\gamma$ in $L^2(\Rd)$ such that
  $\G(\gamma, \Lambda)$ is a frame and  the
  reconstruction formula
  \begin{equation*}
    f = \sum_{\lambda \in \Lambda} \langle f, \pi(\lambda)g \rangle \pi(\lambda) \gamma
  \end{equation*}
  holds for all $f \in L^2(\Rd)$ with unconditional $L^2$-convergence.
  We apply the reconstruction formula to $\pi (z)^*f$ and take the
  inner product with $\pi (z) ^*h$ for $z\in \rdd $ and $h\in \lrd
  $.  Then
  % This justifies interchanging summation and inner product and
  we have 
  \begin{align*}
    \langle f, h \rangle   &=  \langle \pi (z) ^* f, \pi (z)^* h \rangle
    = \sum_{\lambda \in \Lambda} \langle \pi (z) ^*f, \pi(\lambda)g \rangle
    \langle \pi(\lambda) \gamma , \pi (z) ^* h \rangle \\
   & = \sum_{\lambda \in \Lambda} V_gf(z+\lambda)
     \overline{V_{\gamma}h(z+\lambda)} = \Phi (z)
  \end{align*}
  for all $f,h \in L^2(\Rd)$ and all $z \in \Rtd$.  This means that the
  $\Lambda $-periodic function $\Phi $ on the right-hand side is
  constant.

  By Proposition~\ref{sec:short-time-fourier-OrthoRel}\,(ii) the
  Fourier coefficients of the right-hand side are given by
  \begin{equation*}
    \hat{\Phi} (\nu ) = \big( V_gf \cdot \overline{V_\gamma h }\big)\,
    \widehat{}\, (\nu ) =   V_g \gamma(\mu)
    \overline{V_f h(\mu)} \text{\,,} % \qquad \text { for } \mu \in \Lambda
    % ^\circ \, ,
  \end{equation*}
  where $\nu \in \Lambda ^\perp$ and $\mu = \I \nu \in \duall $.
  Since these are the Fourier coefficients of a constant function, they must satisfy
  \begin{equation*}
    \vol(\Lambda)^{-1}  V_g \gamma(\mu)
    \overline{V_f h(\mu)} = \langle f,h \rangle\delta _{\mu ,0} 
    \qquad \forall \mu \in \Lambda^\circ \text{.}
  \end{equation*}

  As this identity holds for all $f, h \in \lrd $, we obtain the
  biorthogonality relation
  \begin{equation*}
    \vol(\Lambda)^{-1}  \langle \gamma , \pi (\mu ) g\rangle = \delta
    _{\mu ,0}  \qquad \forall \mu \in \Lambda ^\circ \text{.}
  \end{equation*}

  By assumption both $\G (g,\Lambda )$ and $\G (\gamma , \Lambda )$
  are frames and thus Bessel sequences, therefore Lemma~\ref{cbessel}
  implies that both $\G (g, \duall ) $ and
  $\G (\gamma , \Lambda ^\circ )$ are Bessel sequences.

  \Rimp{ii}{i}: 
  We use the biorthogonality and read the fundamental identity \eqref{eq:FIGAptw} backwards:
  \begin{equation*}
    \mathrm{vol}\, (\Lambda )\inv \sum _{\mu \in \duall } V_g\gamma
    (\mu ) \overline{V_f h (\mu ) } e^{2\pi i \mu \cdot \I z} = \sum
    _{\lambda \in \Lambda } V_gf(z+\lambda ) \, V_\gamma h (z+\lambda )
    \, .
  \end{equation*}
  Since both $\G (g, \duall )$ and $\G (\gamma , \duall )$ are Bessel
  sequences, Lemma~\ref{cbessel} implies that the Gabor systems
  $\G(g, \Lambda)$ and $\G(\gamma, \Lambda)$ are also Bessel
  sequences.  Furthermore
  $\sum _{\mu \in \duall } |V_g\gamma (\mu )|<\infty $ by the
  biorthogonality
  relation~\eqref{eq:duality-theory-BiorthogonalityRelation}, hence
  all assumptions of Theorem~\ref{sec:fund-ident-gabor-FIGA} are
  satisfied and guarantee that \eqref{eq:FIGAptw} holds
  pointwise. For $z=0$ and $f=h$ we thus obtain
  % Since the left-hand side is an absolutely convergent Fourier series,
  % we do not need the summability kernel, and the identity holds almost
  % everywhere. Thus there exists $z_0\in \rdd $ such that
  \begin{equation*}
    \|f\|\Lnrm^2 = \sum   _{\lambda \in \Lambda } \langle  f , \pi
    (\lambda ) g  \rangle \, \langle \pi (\lambda ) \gamma ,   f \rangle \, .
  \end{equation*}
  Since both sets $\G (g,\Lambda )$ and $\G (\gamma , \Lambda )$ are
  Bessel sequences with Bessel bounds $B_g$ and $B_{\gamma}$
  respectively, the frame inequality for $\G (g,\Lambda )$ is obtained
  as follows:
  \begin{align*}
    \|f \|\Lnrm^4 & \leq \Big( \sum _{\lambda \in \Lambda } |\langle f, \pi
    (\lambda ) g \rangle |^2\Big)  \, \Big( \sum _{\lambda \in \Lambda } |\langle f, \pi
                 (\lambda ) \gamma  \rangle |^2\Big) \\
    &\leq B_{\gamma} \|f \|\Lnrm ^2   \sum _{\lambda \in \Lambda } |\langle f, \pi
    (\lambda ) g \rangle |^2  \leq B_gB_{\gamma}\|f\|\Lnrm^4 \, .
  \end{align*}
% ***********
% Therefore, it remains to show the lower frame
%   inequality for $\G(g, \Lambda)$.

%   The biorthogonality
%   relation~\eqref{eq:duality-theory-BiorthogonalityRelation} implies
%   condition~(A'), hence Janssen's representation of the frame operator, \ie
%   Theorem~\ref{sec:janss-repr-JanssRep}, yields
%   \begin{equation}
%     \label{eq:duality-theory-eq4}
%       S_{g, \gamma} f = \vol(\Lambda)^{-1} \sum_{\mu \in \Lambda^\circ}
%       \langle \gamma, \pi(\mu)g\rangle \pi(\mu)f = f 
%   \end{equation}
%   for all $f \in L^2(\Rd)$.

%   Using \eqref{eq:duality-theory-eq4} and the fact that
%   $D_{\gamma, \Lambda}$ is bounded, we obtain
%   \begin{align}
%     \|f\|^2 = \|S_{g, \gamma} f \|^2 
%      = \|D_{\gamma,\Lambda}( C_{g, \Lambda}f )\|^2 
%     \leq B \sum_{\lambda \in \Lambda} |V_gf(\lambda)|^2
%   \end{align}
%   for all $f \in L^2(\Rd)$, which implies the lower frame inequality.
  
  %                     END (ii) ==> (i)                     %
  %==========================================================%  

  %==========================================================%
  %                   BEGIN (ii) ==> (iii)                   %
  \Rimp{ii}{iii}: By assumption, the Bessel sequences
  $\G(g, \Lambda^\circ)$ and $\G(\gamma, \Lambda^\circ)$ satisfy the
  biorthogonal
  condition~\eqref{eq:duality-theory-BiorthogonalityRelation}, thus
   \begin{equation}
     \label{eq:duality-theory-eq8}
     \langle \pi(\nu) \gamma, \pi(\mu) g \rangle
     = e^{-2\pi i (\mu_2 - \nu_2)\cdot\nu_1} \,  \langle \gamma, \pi(\mu - \nu) g \rangle
     =  \vol(\Lambda) \delta_{\mu - \nu, 0}
     \quad \forall \mu,\nu \in \Lambda^\circ \, .
   \end{equation}
   Define $\tilde{\gamma} := \vol(\Lambda)^{-1}\gamma$, then
   \eqref{eq:duality-theory-eq8} implies that $\G(\tilde{\gamma},
   \Lambda^\circ)$ is a biorthogonal Bessel sequence for $\G(g,
   \Lambda^\circ)$.  This means that  $\G(g, \Lambda^\circ)$ is a Riesz
   sequence. % by Lemma~\ref{sec:frame-theory-CharRieszSequence}.

%   \XX{This probably needs an argument}

  %                     END (ii) ==> (iii)                   %
  %==========================================================% 

  %==========================================================%
  %                   BEGIN (iii) ==> (ii)                   %
   \Rimp{iii}{ii}: By assumption $\G(g, \Lambda^\circ)$ is a Riesz
   sequence, \ie a Riesz basis for its closed linear span, which we
   denote by $\cK := \overline{\vspan}\{\G(g, \Lambda^\circ)\}$. 
   % Theorem~\ref{sec:frame-theory-CharRieszBases} implies that
   By the general  properties of Riesz bases \cite{MR3495345}, there exists a Bessel sequence
   $\{ e_{\nu} : \nu \in \Lambda^\circ \}$ in $\cK $ such that
   \begin{equation*}
     \label{eq:duality-theory-eq10}
     \langle e_{\nu}, \pi (\mu) g \rangle = \delta_{\nu,\mu}
     \qquad \forall \mu, \nu \in \Lambda^\circ \text{.}
   \end{equation*}
   On the other hand, since $\cK $ is invariant with respect to
   $\pi (\nu )$ for all  $\nu \in \duall $, we have that $\pi (\nu )e_0$ is also in $\cK $
   and satisfies the biorthogonality
   \begin{equation*}
     \langle \pi (\nu ) e_0, \pi (\mu ) g \rangle = 
     e^{-2\pi i (\mu_2 - \nu_2)\cdot\nu_1}
     \langle e_0, \pi(\mu - \nu) g \rangle  = \delta_{0, \mu - \nu}
     \, .
   \end{equation*}
   This implies that
   $e_\nu - \pi (\nu )e_0 \in \cK \cap \cK ^\perp = \{0\}$.
   % This yields
   % \begin{equation}
   %   \begin{aligned}
   %     \label{eq:duality-theory-eq9}
   %     C_{g, \Lambda^\circ}(\pi(\nu) e_0)(\mu) &= \langle \pi(\nu) e_0,
   %     \pi(\mu) g \rangle \\
   %     & = e^{-2\pi i (\mu_2 - \nu_2)\cdot\nu_1}
   %     \langle e_0, \pi(\mu - \nu) g \rangle \\
   %     & = \delta_{0, \mu - \nu} \\
   %     & =  \langle e_{\nu}, \pi (\mu) g \rangle = C_{g, \Lambda^\circ}(e_{\nu})(\mu) 
   %   \end{aligned}
   % \end{equation}
   % for all $\mu,\nu \in \Lambda^\circ$.

   % Recall that the coefficient operator of a Riesz basis is injective.
   % Therefore, the restriction of $C_{g, \Lambda^\circ}$ to
   % $\overline{\vspan}\{\G(g, \Lambda^\circ)\}$ is injective.  Since
   % $\overline{\vspan}\{\G(g, \Lambda^\circ)\}$ is closed under
   % time-frequency shifts with respect to $\Lambda^\circ$,
   % equation~\eqref{eq:duality-theory-eq9} implies
   % $e_{\nu} = \pi(\nu) e_0$ for all $\nu \in \Lambda^\circ$.
   
   After the normalization $\gamma := \vol(\Lambda)^{-1}e_0$, the set
   $\G(\gamma, \Lambda^\circ) = \{ \vol(\Lambda)^{-1} e_{\nu} : \nu
   \in \Lambda^\circ \}$
   satisfies the biorthogonality relations~
   \eqref{eq:duality-theory-BiorthogonalityRelation} and is Bessel
   sequence by the properties of Riesz bases.
  %                     END (iii) ==> (ii)                   %
  %==========================================================% 
\end{proof}

The duality theory was foreshadowed by Rieffel's abstract work on
non-com\-mu\-ta\-tive tori~\cite{rieffel88}. The biorthogonality
relations~\eqref{eq:duality-theory-BiorthogonalityRelation} were
discovered by the engineers Wexler and
Raz~\cite{Wexler:1990:DGE:104676.104677} and characterize all possible
dual windows (see Corollary~\ref{sec:duality-theorem-dual-CharDualWindows}).
Janssen~\cite{MR1310658,MR1350700}, Daubechies et al.~\cite{MR1350701}
and Ron-Shen~\cite{MR1460623} made the results of Wexler and Raz rigorous
and further expanded upon them which became the duality theory for
separable lattices.  The theory for general lattices is due to
Feichtinger and Kozek~\cite{MR1601091}.  Recently, Jakobsen and
Lemvig~\cite{MR3419761} formulated density and duality theorems for
Gabor frames along a closed subgroup of the time-frequency plane. We
remark that the duality theory also holds verbatim for general locally
compact Abelian groups admitting a lattice. 

\begin{rem}[Frame bounds and an alternative proof]
  \label{sec:duality-theory-FrameBounds}
  % An alternative way to prove the
  % equivalence \LRimp{i}{iii} can be seen in the proof of the Bessel
  % duality, \ie Theorem~\ref{sec:techn-lemm-bess-BesselDuality}, where we
  % related the frame operator $S_{g, \Lambda}$ to the Gramian
  % $G_{g, \Lambda^\circ} = C_{g, \Lambda^\circ}D_{g, \Lambda^\circ}$ of
  % the dual system.  

  By rewriting Janssen's proof of the duality theory in
  \cite{MR1350700} for general lattices, one can show that
  \begin{equation*}
    A I_{L^2} \leq S_{g, \Lambda} \leq B I_{L^2} \quad  \Longleftrightarrow \quad
     A I_{\ell^2} \leq \vol(\Lambda)^{-1} G_{g, \Lambda^\circ} \leq  B I_{\ell^2}
    \text{.}
  \end{equation*}
  Hence, the family $\G(g,\Lambda)$ is a frame with frame bounds
  $A,B >0$ if and only if $\G(g, \Lambda^\circ)$ is a Riesz sequence
  with bounds $\vol(\Lambda)A, \vol(\Lambda)B >0$ respectively.
\end{rem}

% As we have seen in Section~\ref{sec:gabor-frames}, every Gabor frame
% $\G(g, \Lambda)$ has a canonical dual window
% $\gamma^\circ = S^{-1}_{g, \Lambda}g$ that gives rise to the frame
% expansion
% \begin{equation}
%   \label{eq:duality-theory-FrameExp}
%   f = \sum_{\lambda \in \Lambda} \langle f, \pi(\lambda) \gamma^\circ \rangle \pi(\lambda) g 
%   = \sum_{\lambda \in \Lambda} \langle f, \pi(\lambda) g \rangle \pi(\lambda) \gamma^\circ
% \end{equation}
% for all $f \in L^2(\Rd)$.

% In general, \ie if $\G(g, \Lambda)$ is not a Riesz sequence, the
% coefficients in \eqref{eq:duality-theory-FrameExp} are not unique.
% Hence, there may be other windows $\gamma \in L^2(\Rd)$ that satisfy the
% reconstruction property $S_{g, \gamma} = S_{\gamma, g} = I_{L^2}$.
% This motivates the following definition.

\begin{defi}
  Let $g \in L^2(\Rd)$ and $\G(g, \Lambda)$ be a Bessel sequence.  We
  call $\gamma \in L^2(\Rd)$ a \emph{dual window} for $\G(g,
  \Lambda)$ if $\G(\gamma, \Lambda)$ is a Bessel sequence and the
  reconstruction property 
  $$
  f = \sum _{\lambda\in \Lambda } \langle f, \pi (\lambda )g \rangle
  \pi (\lambda ) \gamma = \sum _{\lambda \in \Lambda } \langle f, \pi
  (\lambda ) \gamma \rangle \pi (\lambda ) g
   $$
  holds for all $f\in \lrd $. %S_{g, \gamma} = S_{\gamma, g}
  % = I_{L^2}$ is satisfied.
\end{defi}

The duality theorem now yields the  following characterization of all  dual windows.%  from the
% proof of the equivalence \LRimp{i}{ii} in
% Theorem~\ref{sec:duality-theory-DualityTheory}.

% From the proof of the equivalence \LRimp{ii}{i}, we can isolate the
% following characterization for dual windows. 

% The statement of the following corollary is already implicit in the
% proof of the equivalence \LRimp{i}{ii} in
% Theorem~\ref{sec:duality-theory-DualityTheory}.

\begin{cor}
  \label{sec:duality-theorem-dual-CharDualWindows}
  Suppose $g, \gamma \in L^2(\Rd)$ and $\Lambda \subseteq \Rtd$ such
  that $\G(g,\Lambda)$ and $\G(\gamma,\Lambda)$ are Bessel sequences.
  Then $\gamma$ is a dual window for $\G(g, \Lambda)$ if and only if
  the Wexler-Raz biorthogonality
  relations~\eqref{eq:duality-theory-BiorthogonalityRelation} are
  satisfied.
\end{cor}

\begin{proof}
  This is simply
  % There is no need to prove anything new, we just put a
  % different focus on
  equivalence \LRimp{i}{ii} of
  Theorem~\ref{sec:duality-theory-DualityTheory}.
\end{proof}

We conclude this section with a characterization of tight Gabor
frames.

\begin{cor}
  \label{sec:duality-theorem-CharTightFrame}
  A Gabor system $\G(g,\Lambda)$ is a tight frame if and only if $\G(g,
  \Lambda^\circ)$ is an orthogonal system.  In this case the frame
  bound satisfies $A = \vol(\Lambda)^{-1}\|g\|\Lnrm^2$.
\end{cor}

\begin{proof}
  If $\G(g,\Lambda)$ is a tight frame, then the frame operator is just
  a multiple of the identity, \ie $S= A I_{L^2}$.  Hence
  the canonical dual window is of the form
  $\gamma = S^{-1}g = \frac{1}{A} g$ and the
  biorthogonality
  relations~\eqref{eq:duality-theory-BiorthogonalityRelation} yield 
  \begin{equation*}
    \label{eq:duality-theory-6}
    \langle g, \pi(\mu) g \rangle
    = A \langle \gamma, \pi(\mu) g \rangle
    = A \vol(\Lambda) \delta_{\mu, 0} \qquad \forall \mu \in \Lambda^\circ \text{.}
  \end{equation*}
  Therefore, $\G(g, \Lambda^\circ)$ is an orthogonal system and in
  particular $A = \vol (\Lambda )\inv \|g\|\Lnrm^2 $.

  % Theorem~\ref{sec:duality-theorem-dual-CharDualWindows}.  Thus
  % $\G(g, \Lambda^\circ)$ is an orthogonal system since
  % \begin{equation}
  %   \label{eq:duality-theory-6}
  %    \langle g, \pi(\mu) g \rangle
  %    = A \langle \gamma, \pi(\mu) g \rangle
  %    = A \vol(\Lambda) \delta_{\mu, 0} \qquad \forall \mu \in \Lambda^\circ \text{.}
  %  \end{equation}
  %  In particular, $A = \vol (\Lambda )\inv \|g\|\Lnrm^2  $.
   
   Conversely, let $\G(g, \Lambda^\circ)$ be an orthogonal system, \ie
   \begin{equation*}
     \langle g, \pi(\mu) g \rangle = \|g\|\Lnrm^2 \delta_{\mu,0}
     \qquad \forall \mu \in \Lambda^\circ \text{.}
   \end{equation*}
   Then by Theorem~\ref{sec:fund-ident-gabor-FIGA} with
   $\gamma =g$, $h=f$ and $z=0$, we obtain
   \begin{equation*}
     \vol (\Lambda )\inv  \|g\|\Lnrm^2 \|f\|\Lnrm^2 
     = \sum _{\lambda \in \Lambda} |\langle f, \pi
     (\lambda )g \rangle |^2 
   \end{equation*}
   %   $g$ satisfies condition~(A) and Janssen's representation of
   % the frame operator, \ie Theorem~\ref{sec:frame-oper-janss-CondaBessel},
   % yields
   % \begin{equation}
   %   S_{g,\Lambda}f = \vol(\Lambda)^{-1} \sum_{\mu \in \Lambda^\circ}
   %   \langle g , \pi(\mu) g \rangle \pi(\mu)f
   %   =  \vol(\Lambda)^{-1}\|g\|^2 f \text{.}
   % \end{equation}
   % Consequently, the frame operator is a multiple of the identity
   % operator on $L^2(\Rd)$
   and thus $\G(g, \Lambda)$ is a tight frame.
   % The claim regarding the frame bound follows immediately from
   % \eqref{eq:duality-theory-6}.
\end{proof}

\section{The Coarse Structure of Gabor Frames}
Many of the fundamental properties of  Gabor frames can be derived
with little effort from the duality theorem. In the following we deal
with the density theorem, the Balian-Low theorem, and the  existence
of Gabor frames. 

\subsection{Density Theorem}
To recover $f$ from the inner products $\langle f, \pi (\lambda
)g\rangle$, we need enough information. The density theorem quantifies
this statement. The density theorem has a long history and has been
proved many times. We refer to Heil's comprehensive
survey~\cite{MR2313431}. Our point is that it follows immediately from
the duality theory.

% The duality theory has an insightful consequence for the density of
% the lattice: If we sample too sparsely in the time-frequency plane,
% reconstruction will be impossible regardless of the window function.

\begin{thm}[Density theorem]
  \label{sec:density-theorem-DensityThm}
  Let $g \in L^2(\Rd)$ and $\Lambda \subseteq \Rtd$ be a lattice.
  Then the following holds:
  \begin{enumerate}[(i)]
  \item If $\G(g,\Lambda)$ is a frame for $\lrd $, then $0< \vol(\Lambda) \leq 1$.
  \item If $\G(g,\Lambda)$ is a Riesz sequence in $\lrd $, then $\vol(\Lambda)
    \geq 1$.
  \item $\G(g,\Lambda)$ is a Riesz basis for $\lrd $ if and only if it is a frame
    and $\vol(\Lambda) = 1$.
  \end{enumerate}
\end{thm}

% Statement \textit{(a)} can be further improved: If $\G(g, \Lambda)$ is
% complete, \ie its linear span is dense in $L^2(\Rd)$, then $0<
% \vol(\Lambda) \leq 1$.  For general lattices, this density result
% regarding completeness is due to Bekka~\cite{MR2078261}. \XX{FORM}

\begin{proof}
  (i) Let $\gamma = S\inv g$ be the canonical dual window of
  $\G (g, \Lambda )$. Then $g$ possesses the following two distinguished
  representations with respect to the frame $\G (g,\Lambda )$:
  \begin{equation*}
    g = 1 \cdot g = \sum _{\lambda \in \Lambda } \langle g, \pi (\lambda
    ) \gamma  \rangle \pi (\lambda ) g \, .
  \end{equation*}
  By the general properties of the dual frame~\cite{DS52}, the latter expansion
  has the coefficients with the minimum $\ell ^2$-norm, therefore
  \begin{equation*}
    \sum _{ \lambda \in \Lambda } | \langle g, \pi (\lambda )\gamma
    \rangle |^2 \leq 1 + \sum _{\lambda \neq 0} 0 = 1 \, .
  \end{equation*}
  Consequently, with the
  biorthogonality~\eqref{eq:duality-theory-BiorthogonalityRelation}
  (in fact, we only need the condition for $\mu = 0$) we obtain
  \begin{equation} \label{eq:ti8}
    \vol (\Lambda ) ^2 = \langle g, \gamma  \rangle ^2 \le \sum _{ \lambda
      \in \Lambda } | \langle g, \pi (\lambda )\gamma \rangle |^2 \leq 1
    \, , 
  \end{equation}
  which is the density theorem.

  (ii)  The volume of the adjoint lattice is
  $\vol(\Lambda^\circ) = \vol(\Lambda)^{-1}$.  Therefore, the claim is
  equivalent to (i) by
  Theorem~\ref{sec:duality-theory-DualityTheory}.

  (iii)  A Riesz basis is a Riesz sequence that is complete in
  the Hilbert space, and therefore is also a frame. Consequently, both (i)
  and (ii) apply and thus $\vol(\Lambda)=1$. 

  Conversely, if $\G(g, \Lambda)$ is a frame with $\vol(\Lambda)=1$,
  then we have equality in \eqref{eq:ti8} and thus $\langle g, \pi
  (\lambda ) g \rangle = \delta _{\lambda , 0}$ for $\lambda \in
  \Lambda $. Since the  Gabor system $\G (\gamma , \Lambda )$ is a
  Bessel sequence and biorthogonal to $\G (g,\Lambda )$, we deduce
  that $\G (g,\Lambda )$ is a Riesz sequence, and by the
  assumed completeness it is a Riesz basis for $\lrd $. 
 \end{proof}

 The above proof of the density theorem is due to Janssen~\cite{MR1310658}. 

\subsection{Existence of Gabor Frames for sufficiently dense lattices}

In the early treatments of Gabor frames one finds many qualitative
statements that assert the existence of Gabor frames. Typically they
claim that for a window function $g\in \lrd $ with ``sufficient'' decay
and smoothness and for a ``sufficiently dense'' lattice $\Lambda $ the
Gabor system $\G (g, \Lambda )$ is a frame for $\lrd $. For a sample
of results we refer to ~\cite{daubechies92,MR1021139,walnut93}. In this section
we derive such a qualitative result as a consequence of the duality
theorem.

We will measure decay and smoothness by means of \tf\ concentration as
follows: we say that $g$ belongs to the \modsp\ $M^\infty _{v_s}(\rd
)$ if
\begin{equation*}
  \label{eq:c16}
  |V_gg(z) | \leq C (1+|z|)^{-s} \qquad \forall z\in \rdd \,
  .
\end{equation*}
This is not the standard definition of the \modsp , but it is the most
convenient definition for our purpose.  A systematic
exposition of modulation spaces is contained in~\cite{MR1843717}.

To quantify the density of a
lattice $\Lambda = A \zdd $, we set simply
$$
\|\Lambda \| = \|A \|_{op} \, ,
$$
with the understanding that this definition is highly ambiguous and
depends more on the choice of a basis $A$ for the lattice than on the
lattice itself.

\begin{thm} \label{t:ex}
  Assume that $g\in M^\infty _{v_s}(\rd ) $ for some $s>2d$. Then
  there exists a $\delta _0$ depending on $g$ such that for every
  lattice $\Lambda = A \zdd $ with $\|A\|_{op} < \delta _0$  the Gabor
  system $\G (g, \Lambda )$ is a frame for $\lrd $.  
\end{thm}
In other words, there exists a sufficiently small neighborhood $V$  of the
zero matrix such that $\G (g, A \zdd ) $ is a frame for every $A\in V
$. 

\begin{proof}
  Invariably, qualitative existence theorems in Gabor analysis (and
  more generally in sampling theory) use the fact that an operator
  that is close enough to the identity operator is 
  invertible. For this proof, we use the duality theorem and show that
  on the adjoint lattice, which is sufficiently sparse, the Gramian
  matrix is diagonally dominant and therefore invertible.

  Without loss of generality we assume that $\|g\|\Lnrm=1$, then the
  Gramian can be written as  $G =
  \mathrm{I} + R$, where $G_{\mu , \nu } = \langle \pi (\nu ) g, \pi
  (\mu ) g \rangle$ and $R$ is the off-diagonal part of $G$. 

  We now make the following observations about $R$: 

  (i) If  $\|A\|_{op} = \delta  $ and $\mu = \I (A^T)\inv k \in \duall $, then
  $|k| = | A^T \I \inv \I (A^T)\inv k | \leq \|A\|_{op} |\mu | =
  \delta 
  |\mu |$ and
  therefore
  $$
  (1+|\mu |)^{-s} \leq (1+\delta \inv |k|)^{-s} \, .
  $$
  
  (ii) By applying a simplified version of Schur's test to the
  self-adjoint operator $R$, the operator norm of $R$ can estimated by
  \begin{align}
    \|R\|_{op} &\leq \sup _{\mu \in \duall } \sum _{\nu \neq \mu } |\langle
                 \pi (\nu ) g, \pi (\mu ) g \rangle |   \notag \\
               &= \sup _{\mu \in \duall } \sum _{\nu \neq \mu } |\langle
                 g, \pi (\mu -\nu )g \rangle | \notag \\
               &\leq \sum_{ \mu \neq 0} (1+|\mu |)^{-s} \label{mui} \\
               &\leq \sum_{ \substack{k\in \zd \\ k \neq 0}} (1+\delta \inv |k |)^{-s} := \phi
                 (\delta ) \, . \notag 
  \end{align}

  (iii) Since $s>2d$, $\phi (\delta ) $ is finite for all $\delta >0$,
  and $\phi $ is a continuous, increasing function that satisfies
  $$
  \lim _{\delta\to 0+} \phi (\delta ) = 0 \, .
  $$
  Consequently, there is a $\delta _0$ such that $\phi (\delta_0)=1$.
  For $\delta < \delta _0$ we then obtain that
  $$
  \|R\|_{op} \leq \phi (\delta ) < 1 \, ,
  $$
  therefore $G$ is invertible on $\ell ^2(\duall )$. This means that
  $\G (g,\duall )$ is a Riesz sequence, and by duality 
  $\G (g,\Lambda )$ is a frame, whenever the matrix $A$ defining
  $\Lambda =A\zdd $ satisfies $\|A\|_{op} < \delta _0$.
\end{proof}

The above proof highlights the role of the duality theorem in the
qualitative existence proof. By emphasizing some technicalities about
\modsp s, one may prove a slightly more general version of the
existence theorem.
We say that $g$ belongs to the \modsp\ $M^1(\rd )$ if
$$
\intrdd |\langle g, \pi (z) g\rangle | \, \dif z < \infty \, .
$$
The proof of Theorem~\ref{t:ex} can be extended to yield the following
result. 
\begin{thm}[~\cite{MR1021139,fg92chui}]
  Assume that $g\in M^1(\rd ) $. Then
  there exists a $\delta _0$ depending on $g$ such that for every
  lattice $\Lambda = A \zdd $ with $\|A\|_{op} < \delta _0$  the Gabor
  system $\G (g, \Lambda )$ is a frame for $\lrd $.  
\end{thm}

These existence results are complemented by an important theorem of
Bekka~\cite{bekka04}: \emph{For every lattice $\Lambda $ with
  $\mathrm{vol} \, (\Lambda )\leq 1$, there exists a window $g\in \lrd $
  such that $\G (g,\Lambda )$ is a frame.}

\subsection{Balian-Low Theorem}

The Balian-Low theorem (BLT) states that for a window with  a mild
decay in time-frequency the necessary density condition must be
strict. In the standard formulation,  the window a Gabor frame $\G (g,\Lambda )$ at
the critical density $\mathrm{vol}\, ( \Lambda ) = 1$ lacks time-frequency
localization. We refer to the surveys~\cite{BHW95,CP06} for a detailed
discussion of the Balian-Low phenomenon in dimension $1$. For higher
dimensions and arbitrary lattices, the BLT follows from an important
deformation result of Feichtinger and Kaiblinger~\cite{FK04} with useful
subsequent improvements in~\cite{MR3192621,MR3336091}.

\begin{thm}\label{t:blt}
  Assume that $g\in M^\infty _{v_s}(\rd )$ for some $s>2d$ and that
  $\G (g, \Lambda )$ is a frame for $\lrd $. Then there exists an
  $\epsilon _0>0$ such that $\G (g, (1+\tau )\Lambda )$ is a
  frame for every $\tau $ with $ |\tau | < \epsilon _0$. 
\end{thm}

\begin{proof}
  We only give the proof idea and indicate where the duality theorem
  enters. Let $\tilde{ \Lambda } = (1+\tau ) \Lambda $, then its
  adjoint lattice is $\tilde{\Lambda }^\circ = (1+\tau )\inv \duall
  $. If $\G (g,\Lambda )$ is a frame, then the Gramian operator $G=
  G_{g,\duall }$ is invertible on $\ell ^2(\duall )$. Set $\rho =
  (1+\tau )\inv $ and  we consider
  the cross-Gramian operator $\tilde{G}^\rho$ with entries
  \begin{equation*}
    \tilde{G}^\rho_{\mu \nu } = \langle \pi (\rho \nu ) g, \pi (\mu )g
    \rangle \qquad \mu , \nu \in \duall \, .
  \end{equation*}
  We argue that
  \begin{equation}
    \label{eq:m}
    \lim _{\rho \to 1} \|\tilde{G}^\rho  - G \|_{op} = 0 \, .
  \end{equation}
  This implies that for $|\rho -1|< \epsilon _0$ for some
  $\epsilon _0$ the cross-Gramian operator $\tilde{ G}^\rho $ is
  invertible on $\ell ^2(\duall )$. Now a perturbation result for
  Riesz bases that goes back to Paley-Wiener (see,
  e.g.,~\cite{MR3495345}) implies that
  $\G (g, (1+\tau )\inv \Lambda )$ is a Riesz sequence, and by the
  duality theorem  $\G(g, (1+\tau ) \Lambda )$ is a frame.
  
  The proof of \eqref{eq:m} is similar to the proof of
  Theorem~\ref{t:ex}. We apply Schur's test to estimate the operator
  norm of $\tilde{G}^\rho -G$.
  Given $\delta >0$, we may choose $R>0$ such that
  $$
  \sum _{\mu : |\mu -\nu | >R} |\tilde{G}^\rho _{\mu \nu } - G_{\mu
    \nu }| < \delta  /2    \qquad \text{ for all } \nu \in \duall \,
  \, \text{ and } \, 
  1/2 < \rho <2 \, ,  
  $$
  and likewise
  $\sum _{\nu : |\mu - \nu |> R} |\tilde{G}^\rho _{\mu \nu } - G_{\mu
    \nu }| < \delta /2 $ for all $\mu \in \duall$.  As in \eqref{mui},
  this is possible because $g\in M^\infty _{v_s}(\Rd)$ guarantees the
  off-diagonal decay of $G$ and $\tilde{G}^\rho $.

  Next, we choose $\epsilon _0> 0$ such that for
  $|\rho -1|<\epsilon _0$
  \begin{equation*}
    \sum _{\mu : |\mu -\nu | \leq R} |\tilde{G}^\rho _{\mu \nu } - G_{\mu
      \nu }| = \sum _{\mu : |\mu -\nu | \leq R} |\langle \pi (\rho \nu
    )g - \pi (\nu )g, \pi (\mu )g \rangle | < \delta /2 
  \end{equation*}
  and 
  $\sum _{\nu : |\mu -\nu | \leq R} |\tilde{G}^\rho _{\mu \nu } -
  G_{\mu \nu }| < \delta /2 $. Combining both estimates yields
  $\|\tilde{G}^\rho - G \|_{op} < \delta $.
\end{proof}

Again, the optimal assumption on $g$ in Theorem~\ref{t:blt} is that it belongs to
$M^1(\rd )$.
 
\begin{cor}
  Assume that $\vol (\Lambda ) = 1$ and $\G (g, \Lambda )$ is a frame
  for $\lrd $. Then $g\not \in M^\infty _{v_s}(\rd )$ for all $s>2d$.
\end{cor}
\begin{proof}
  If  $g\in M^\infty _{v_s}$ and $\G (g, \Lambda )$ were a frame for
  some lattice $\Lambda $ with $\vol (\Lambda ) = 1$, then
  by Theorem~\ref{t:blt} the   Gabor system $\G (g, (1+\tau )\Lambda ) $
  would also be a frame for some $\tau >0$. But $\vol \big( (1+\tau )\Lambda
  \big) = (1+\tau )^{2d} \vol (\Lambda ) >1$, and this contradicts the
  density theorem. 
\end{proof}

\subsection{The Coarse Structure of Gabor Frames and Gabor Riesz
  Sequences} 

One of the principal questions of Gabor analysis is the question
under which conditions on a window $g$ and a lattice $\Lambda $ the
Gabor system $\G (g,\Lambda )$ is a frame for $\lrd $ or a Riesz
sequence in $\lrd $.  To formalize this, we define the full
frame set $\F _{\mathrm{full}} (g)$ of $g$ to be the set of all
lattices $\Lambda $ such that $\G (g,\Lambda )$ is a frame and the
reduced frame set $\F (g)$ to be the set of all rectangular lattices
$\alpha \zd \times \beta \zd$ such that
$\G(g, \alpha \zd \times \beta \zd )$ is a frame. Formally,
 \begin{align*}
  \F _{\mathrm{full}} (g) &= \{\Lambda  \subseteq \rdd  \,\, \mathrm{
                            lattice }\, : \G(g, \Lambda)
   \text{ is a frame }\} \\
 \F(g) &= \{(\alpha,\beta) \subseteq  \R ^2 _+ : \G(g, \alpha \zd \times \beta
         \zd ) \text{ is a frame }\} \, . 
 \end{align*}

We summarize the results of the previous sections in the main result
about the coarse structure of Gabor frames, i.e., results that hold
for arbitrary Gabor systems over a lattice.

\begin{thm} \label{corase} If $g \in M ^\infty _{v_s} (\rd )$ for some
  $s>2d$ or  in $ M^1(\rd )$, then $\F _{\mathrm{full}} (g)$ is an open
  subset of $\{\Lambda : \vol(\Lambda ) < 1\}$ and contains a
  neighborhood of $0$.

Likewise, $\F(g)$ is open in $\{(\alpha ,\beta) \in  \R ^2 _+ :
\alpha\beta < 1\}$  and contains a neighborhood
of $(0, 0)$ in $\R ^2 _+$.
\end{thm}

Theorem~\ref{corase} should not be underestimated. It 
compresses the efforts of dozens of articles into a single
statement. It contains the existence of Gabor frames,  the density
theorem, and the  Balian-Low theorem. For each result there are now several
different proofs (with subtle differences in the hypotheses) and many
ramifications. What is perhaps new  is the close
connection of the coarse structure of Gabor frames to the duality
theory. % Perhaps the main merit of this chapter is the systematic 
% derivation of the main statements from the duality theory and 

\section{The Criterion of Janssen, Ron and Shen for Rectangular Lattices}
\label{sec:criterion-ron-shen}

In this and the following section, we consider rectangular lattices of
the form $\Lambda = \alpha \Z^d \times \beta \Z^d$ for
$\alpha, \beta >0$.  Observe that the adjoint of such a lattice is
\begin{equation*}
  \Lambda^\circ = \I \begin{pmatrix}
    \frac{1}{\alpha} I_d & 0 \\
    0 & \frac{1}{\beta} I_d
  \end{pmatrix} \Z^{2d}
  = \tfrac{1}{\beta} \Z^d \times \tfrac{1}{\alpha} \Z^d \text{,}
\end{equation*}
and is  again a rectangular lattice.

For convenience, we denote
$\G(g, \alpha, \beta) := \G(g, \alpha \Z^d \times \beta \Z^d)$. %  and
% analogously the corresponding frame operators by
% $S_{g, \alpha , \beta}, C_{g, \alpha , \beta}, D_{g, \alpha , \beta}$.

\begin{defi}
  Let $g \in L^2(\Rd)$ and $\Lambda = \alpha \Z^d \times \beta \Z^d$
  with $\alpha, \beta >0$.  The
  \emph{pre-Gramian matrix} $P(x)$ is defined
  by
  \begin{equation*}
    P(x)_{j,k} = \overline{g}\big(x+\alpha j - \tfrac{k}{\beta}\big)
    \qquad \forall j,k \in \Z^d \, , 
  \end{equation*}
  and the \emph{Ron-Shen matrix}
  $\RS(x) := P(x)^* P(x)$ has the entries
  \begin{equation*}
    \RS(x)_{k,l} = \sum_{j \in \Z^d}    g\big(x+\alpha j - \tfrac{k}{\beta}\big)
     \overline{g}\big(x+\alpha j - \tfrac{l}{\beta}\big)
     \qquad \forall k,l \in \Z^d \text{.}
  \end{equation*}
\end{defi}

\begin{thm}
  \label{sec:rectangular-lattices-Char}
  Let $g \in L^2(\Rd)$ and $\Lambda = \alpha \Z^d \times \beta \Z^d$
  with $\alpha, \beta >0$ be a rectangular lattice. Then the following
  are equivalent:
  \begin{enumerate}[(i)]
  \item $\G(g, \alpha, \beta)$ is a frame for $L^2(\Rd)$. \label{item:frame}
    
  \item $\G(g, \alpha, \beta)$ is a Bessel sequence and there exists a
    dual window $\gamma \in L^2(\Rd)$ such that
    $\G (\gamma , \alpha, \beta )$ is a Bessel sequence satisfying
    \begin{equation}
      \label{eq:j9}
      \sum _{ j\in \Z^d} \gamma (x+\alpha j ) \bar{g} \big(x+\alpha j -
      \tfrac{k}{\beta}\big) = \beta^d \delta _{k,0} \qquad \forall k\in
      \zd \text{ and a.e. }x \in \Rd \, .
    \end{equation} \label{item:dual window}

  \item $\G(g, \frac{1}{\beta}, \frac{1}{\alpha})$ is
    a Riesz sequence for $L^2(\Rd)$. \label{item:riesz}

  \item There exist positive constants $A,B > 0$ such that for all
    $c \in \ell^2(\Z^d)$ and almost all $x \in \Rd$
    \begin{equation}
      \label{eq:rectangular-lattices-eq1}
      A \|c\|\lnrm^2 \leq \sum_{j \in \Z^d} \Big| \sum_{k \in \Z^d} c_k
        \overline{g}\big(x+\alpha j - \tfrac{k}{\beta}\big)\Big|^2 
        \leq B \|c\|\lnrm^2 \text{.}
    \end{equation} \label{item:ron-shen}
    
  \item There exist positive constants $A,B > 0$ such that the
    spectrum of almost every Ron-Shen matrix is contained in the interval
    $[A,B]$.  This means
    \begin{equation*}
      \sigma(\RS(x)) \subseteq [A,B] \qquad \text{for a.e. } x \in \Rd \text{.}
    \end{equation*} \label{item:spectrum}

  \item The set of pre-Gramians $\{P(x)\}$ is uniformly bounded on
    $\ell^2(\Z^d)$ and has a set of uniformly bounded left-inverses.
    This means that there exist
    $\Gamma(x): \ell ^2 (\zd ) \to \ell ^2(\zd)$ such that
    \begin{align*}
      \Gamma(x) P(x) = I_{\ell^2(\Z^d)} \qquad &\text{for a.e. } x \in \Rd \text{,} \\
      \|\Gamma(x)\| \leq C \qquad  &\text{for a.e. } x \in \Rd
                                     \text{.}
    \end{align*} \label{item:groe-stoe}

  %%%%%%%%%%%%%%%%%%%%%%%%%%%%%%%%%%%%%%%%%%%%%%%%%%%%%%%%%%%% 
  %           COUNTER FOR LATER CHARACTERIZATIONS            
  \newcounter{CountRectChar}
  \setcounter{CountRectChar}{\value{enumi}}
  %%%%%%%%%%%%%%%%%%%%%%%%%%%%%%%%%%%%%%%%%%%%%%%%%%%%%%%%%%%% 
  \end{enumerate}
\end{thm}

% The first two characterizations is just a special case of
% Theorem~\ref{sec:duality-theory-DualityTheory}.

% The proof of the equivalences
% $(iii) \Leftrightarrow (iv) \Leftrightarrow (v)$ follows the 
% survey article \cite{MR3232589}.  The proof of the last criterion is
% the original by Gr\"ochenig and St\"ockler~\cite{MR3053565}.

\begin{proof}
  The equivalence of \eqref{item:frame} and \eqref{item:riesz} is
  Theorem~\ref{sec:duality-theory-DualityTheory}.  The equivalence of
  conditions \eqref{item:ron-shen}, \eqref{item:spectrum} and
  \eqref{item:groe-stoe} is mainly of linguistic nature, the
  mathematical content is in the equivalence
  $\eqref{item:riesz} \Leftrightarrow \eqref{item:ron-shen}$.

  \LRimp{\ref{item:ron-shen}}{\ref{item:spectrum}}: For all sequences
  $c \in \ell^2(\Z^d)$, we have
  \begin{equation*}
    \sum_{j \in \Z^d} \Big| \sum_{k \in \Z^d} c_k \overline{g}\big(x+ \alpha j -
        \tfrac{k}{\beta}\big)\Big|^2 = 
    \langle P(x) c, P(x) c \rangle = \langle \RS(x)c, c \rangle \text{.}
  \end{equation*}
  Hence, inequality~\eqref{eq:rectangular-lattices-eq1} becomes 
  \begin{equation*}
    A \|c\|\lnrm^2 \leq \langle \RS(x) c, c \rangle \leq B \|c\|\lnrm^2
    \qquad \forall c \in \ell^2(\Z^d) \text{,}
  \end{equation*}
  for almost all $x \in \Rd$, which is equivalent to $\sigma(\RS(x))
  \subseteq [A,B]$ for almost all $x \in \Rd$. 

  \Rimp{\ref{item:ron-shen}}{\ref{item:riesz}}: Let
  $c\in \ell ^2(\zdd )$ be a finite sequence and
  $f = \sum _{k,l \in \zd } c_{k,l}
  M_{\frac{l}{\alpha}}T_{\frac{k}{\beta}} g$.
  For fixed $k$ the sum over $l$ is a trigonometric polynomial
  $$
  p_k(x) := \sum_{l \in \Z^d} c_{k,l} e^{2 \pi i \frac{l}{\alpha}
    \cdot x}
  $$ 
  with period $\alpha $ in each coordinate, and its $L^2$-norm over a
  period $Q_{\alpha} : = [0,\alpha]^d$ given by
  \begin{equation*}
    \int_{Q_{\alpha}} |p_k(x)|^2 \dif x = \alpha^d \sum_{l \in \Z^d}
    |c_{k,l}|^2 \, .
  \end{equation*}

% We need to show that $D_{g, \frac{1}{\beta},
%     \frac{1}{\alpha}}: \ell^2(\Z^{2d}) \to L^2(\Rd)$ is bounded from
%   above and below.

%   Due to the separable structure of the lattice, we can rewrite the
%   synthesis operator as
%   \begin{equation}
%     (D_{g, \frac{1}{\beta},
%     \frac{1}{\alpha}}c)(x)  = \sum_{k,l \in \Z^d} c_{k,l} M_{\frac{l}{\alpha}}T_{\frac{k}{\beta}} g(x)
%     = \sum_{k \in \Z^d} \bigg( \sum_{l \in \Z^d} c_{k,l} e^{2 \pi i \frac{l}{\alpha} \cdot x} \bigg)
%     T_{\frac{k}{\beta}}g(x) \text{,}
%   \end{equation}
%   where $p_k(x) := \sum_{l \in \Z^d} c_{k,l} e^{2 \pi i
%     \frac{l}{\alpha} \cdot x}$ is a periodic Fourier series on
%   $Q_{\alpha} : = [0,\alpha]^d$ with square-summable coefficients
%   $(c_{k,l})_{l \in \Z^d}$.  Consequently, the Fourier series $p_k$ is
%   in $L^2(Q_{\alpha})$ and we have
%   \begin{equation}
%     \int_{Q_{\alpha}} |p_k(x)|^2 \dif x = \alpha^d \sum_{l \in \Z^d} |c_{k,l}|^2
%   \end{equation}
%   by Plancherel's theorem.
  To calculate  the $L^2$-norm of $f$ we use the periodization trick and
  obtain
  % We periodize the $L^2$-norm of the synthesis operator with respect
  % to $Q_{\alpha}$ as follows:
  \begin{align*}
%    \| D_{g, \frac{1}{\beta}, \frac{1}{\alpha}} c \|_{L^2}^2
%    & = \| \sum_{k,l \in \Z^d} c_{k,l} M_{\frac{l}{\alpha}}T_{\frac{k}{\beta}} g \|_{L^2}^2 \\
    \|f\|\Lnrm^2    & = \Big\| \sum_{k \in \Z^d} p_k \cdot T_{\frac{k}{\beta}}g \Big\|\Lnrm^2 \\
     & = \int_{\Rd} \Big| \sum_{k \in \Z^d} p_k(x)  g\big(x-\tfrac{k}{\beta}\big)\Big|^2 \dif x \\
     & = \int_{Q_{\alpha}} \sum_{j \in \Z^d} \Big| \sum_{k \in \Z^d} p_k(x)
       g\big(x+ \alpha j-\tfrac{k}{\beta}\big)\Big|^2 \dif x \text{.}
        \label{eq:rectangular-lattices-eq2}
  \end{align*}
  % Observe that $\big(p_k(x)\big)_{k \in \Z^d}$ is in $\ell^2(\Z^d)$ for almost
  % all $x \in \Rd$ since 
  % \begin{equation}
  %   \int_{Q_{\alpha}} \sum_{k \in \Z^d} |p_k(x)|^2 \dif x = \sum_{k \in \Z^d}
  %   \int_{Q_{\alpha}} |p_k(x)|^2 \dif x = \alpha^d \sum_{k,l \in \Z^d} |c_{k,l}|^2 < \infty
  %   \text{.}
  % \end{equation}
  Next, for every $x\in Q_\alpha $ we apply assumption
  \eqref{eq:rectangular-lattices-eq1} to the integrand
  % with respect to the sequence $\big(p_k(x)\big)$
  and obtain 
  \begin{align*}
%     \| D_{g, \frac{1}{\beta}, \frac{1}{\alpha}} c \|_{L^2}^2 
    \|f\|\Lnrm^2    & \geq \int_{Q_\alpha} A \sum_{k \in \Z^d} |p_k(x)|^2 \dif x \\
    % & = A\sum_{k \in \Z^d} \int_{Q_{\alpha}} |p_k(x)|^2 \dif x \\
     & = \alpha^d A \sum_{k,l \in \Z^d} |c_{k,l}|^2 = \alpha^d A
     \|c\|\lnrm^2 \, .
  \end{align*}
  for all finite sequences  $c \in \ell^2(\Z^{2d})$.  The upper bound
  follows analogously, and thus $\G (g, \tfrac{1}{\beta},
  \tfrac{1}{\alpha})$ is a Riesz sequence.  % By density, \eqref{} holds for all $c\in \ell
  % ^2(\zdd )$. 

  \Rimp{\ref{item:riesz}}{\ref{item:ron-shen}}: By
  assumption, % the synthesis operator
  % $D_{g, \frac{1}{\beta}, \frac{1}{\alpha}}$ is bounded from above and
  % below.  This means we have
  \begin{equation*}
    A \|c\|\lnrm^2 \leq \Big\| \sum_{k,l \in \Z^d} c_{k,l}
    M_{\frac{l}{\alpha}}T_{\frac{k}{\beta}} g \Big\|\Lnrm^2
    \leq B \|c\|\lnrm^2 \text{}
  \end{equation*}
  for all $c \in \ell^2(\Z^{2d})$. We apply this fact to sequences $c$
  of the form $c_{k,l} :=
  a_k b_l$    for $a, b \in \ell^2(\Z^d)$. 
  % For $a, b \in \ell^2(\Z^d)$, define the sequence $c$ by $c_{k,l} :=
  % a_k b_l$.  Then $c \in \ell^2(\Z^{2d})$ since
  Then
  $\|c\|_{\ell^2(\Z^{2d})}^2= \|a\|_{\ell^2(\Z^d)}^2
  \|b\|_{\ell^2(\Z^d)}^2 $.
 
  % \begin{equation}
  %   \|c\|_{\ell^2(\Z^{2d})}^2 = \sum_{k,l \in Z^d} |c_{k,l}|^2
  %   = \sum_{k,l \in \Z^d} |a_k|^2 |b_l|^2 = \|a\|_{\ell^2(\Z^d)}^2 \|b\|_{\ell^2(\Z^d)}^2 \text{.}
  % \end{equation}
  
  Every $p \in L^2(Q_{\alpha})$ can be written as Fourier series
  $p(x) = \sum_{l \in \Z^d} b_l e^{2 \pi i l \cdot \frac{x}{\alpha}}$
  with coefficients $b \in \ell^2(\Z^d)$.  Hence, we obtain for
  arbitrary $a \in \ell^2(\Z^d)$ and $p \in L^2(Q_{\alpha})$
  \begin{equation}
  \begin{aligned}
    \frac{A}{\alpha^d} \|a\|_{\ell^2(\Z^d)}^2 \int_{Q_{\alpha}} &|p(x)|^2 \dif x 
    = A \|a\|_{\ell^2(\Z^d)}^2 \|b\|_{\ell^2(\Z^d)}^2  = A \|c\|_{\ell^2(\Z^{2d})}^2 \\
    & \leq  %\| D_{g, \frac{1}{\beta}, \frac{1}{\alpha}} c
            %\|_{L^2}^2 = 
    \Big\| \sum_{k,l \in \Z^d} a_k b_l M_{\frac{l}{\alpha}}T_{\frac{k}{\beta}} g \Big\|\Lnrm^2 \\
      % & = \int_{\Rd} \Big|\sum_{k \in \Z^d} a_k \sum_{l \in \Z^d}b_l
      % e^{2 \pi i l \cdot \frac{x}{\alpha}} g\Big(x- \frac{k}{\beta}\Big)\Big|^2 \dif x \\
    & = \int_{\Rd} |p(x)|^2 \Big| \sum_{k \in \Z^d} a_k
    g\big(x - \tfrac{k}{\beta}\big)\Big|^2 \dif x \label{eq:rectangular-lattices-eq11} \\
    & = \int_{Q_{\alpha}} \sum_{j \in \Z^d} |p(x+ \alpha j )|^2 \Big| \sum_{k \in \Z^d} a_k
    g\big(x+\alpha j - \tfrac{k}{\beta}\big)\Big|^2 \dif x \\
    & = \int_{Q_{\alpha}} |p(x)|^2 \sum_{j \in \Z^d} \Big| \sum_{k \in \Z^d} a_k
    g\big(x+\alpha j - \tfrac{k}{\beta}\big)\Big|^2 \dif x \text{.}
  \end{aligned}
  \end{equation}

  Since $L^2(Q_{\alpha})$ contains all characteristic functions of
  measurable subsets in $Q_{\alpha}$,
  \eqref{eq:rectangular-lattices-eq11} implies 
  \begin{equation*}
    \frac{A}{\alpha^d} \|a\|\lnrm^2 \leq \sum_{j \in \Z^d} \Big| \sum_{k \in \Z^d} a_k
        g\big(x+\alpha j - \tfrac{k}{\beta}\big)\Big|^2
        \qquad \text{for a.e. } x \in \Rd 
  \end{equation*}
  for all $a \in \ell^2(\Z^d)$.  The upper bound follows analogously.

  \Rimp{\ref{item:spectrum}}{\ref{item:groe-stoe}}: Suppose that the
  spectrum of almost all $\RS(x)$ is contained in the interval $[A,B]$
  for some positive constants $A,B > 0$. Then the set of pre-Gramians
  is uniformly bounded by $B^{1/2}$ since
  $\RS(x) = P(x)^*P(x)$.

  As $\RS(x)$ is invertible, we may define the pseudo-inverse 
  $\Gamma(x) := \RS(x)^{-1}P(x)^*$.  Then
  \begin{equation*}
    \Gamma(x)P(x) = I_{\ell^2(\Z^d)}
  \end{equation*}
  and
  \begin{equation*}
    \|\Gamma(x)\| \leq \|\RS(x)^{-1}\| \|P(x)\| \leq A^{-1}B^{1/2} \text{.}
  \end{equation*}

  \Rimp{\ref{item:groe-stoe}}{\ref{item:spectrum}}: By assumption,
  every $P(x)$ possesses a left inverse $\Gamma (x) $ with control of
  the operator norm.  This implies
  \begin{equation*}
    \begin{aligned}
      \label{eq:rectangular-lattices-eq3}
      \|c\|\lnrm^2 &= \|\Gamma(x)P(x)c\|\lnrm^2
      \leq  \|\Gamma(x)\|^2 \|P(x)c\|\lnrm^2 \\
      &\leq C^2 \langle \RS(x) c, c \rangle \leq C^2 \|P(x)\|^2
      \|c\|\lnrm^2 \leq C^2 D^2 \|c\|\lnrm^2 \, .
    \end{aligned}
  \end{equation*}
  for all $c \in \ell^2(\Z^{2d})$ and almost all $x \in \Rd$.  This is
  \eqref{item:spectrum}.

  \LRimp{\ref{item:dual window}}{\ref{item:riesz}} and
  \LRimp{\ref{item:dual window}}{\ref{item:groe-stoe}}: Condition
  \eqref{item:dual window} can be understood as as explicit version of
  \eqref{item:groe-stoe}. Alternatively, it is a slight reformulation
  of the biorthogonality condition
  \eqref{eq:duality-theory-BiorthogonalityRelation}, once again with
  the Poisson summation formula:
  \begin{equation*}
    \sum _{ j\in \Z^d} \gamma (x+\alpha j ) \bar{g} \big(x+\alpha j -
    \tfrac{k}{\beta}\big) =  \frac{1}{\alpha^d }\sum _{j\in \zd } \langle
    \gamma , M_{\frac{j}{\alpha }} T_{\frac{k}{\beta} } g\rangle e^{2\pi i  \frac{j}{\beta} \cdot x }
    = \beta^d \delta _{k,0} \, .
  \end{equation*}
  % $P(x)$ is also uniformly bounded for almost all $x \in \Rd$,
  % inequality~\eqref{eq:rectangular-lattices-eq3} is equivalent to the
  % Ron-Shen criterion~\eqref{eq:rectangular-lattices-eq1}.
\end{proof}

The formulation \eqref{eq:j9} of the biorthogonality is due to
Janssen~\cite{MR1601115}.  Conditions \eqref{item:ron-shen} and
\eqref{item:spectrum} were discovered by Ron and
Shen~\cite{MR1460623}.  The criterion \eqref{item:groe-stoe} is
from~\cite{MR3053565}.

The results of Ron and Shen are more general and hold for
\emph{separable lattices} of the form $P \Z^d \times Q \Z^d$ with
invertible, real-valued $d \times d$ matrices $P, Q$.  In this setting,
Theorem~\ref{sec:rectangular-lattices-Char} holds with the appropriate
modifications (just replace the scalar-multiplication with
$\alpha, \beta, 1 / \alpha, 1 / \beta$ by the matrix-vector
multiplication with $P$, $Q$, $P^{-1}$, $Q^{-1}$ and use appropriate
fundamental domains).
  % The proofs remain identical, but the resulting Fourier series are
  % now periodic on the fundamental domain of the lattice $P \Z^d$
  % instead of $Q_{\alpha} =[0,\alpha]^d$.

Condition \eqref{item:ron-shen} has been the master tool of Janssen in
his work on exponential windows~\cite{MR1964306} and ``Zak transforms
with few zeros''~\cite{MR1955931}.  The construction of a dual window
was used by Janssen~\cite{MR1310658} to give a signal-analytic proof
of the Theorem of Lyubarski and Seip. Recently, the biorthogonality
condition for the dual window was used successfully in the analysis of
totally positive windows of finite type~\cite{MR3053565}.
Christensen \etal~\cite{MR2563261} have used \eqref{eq:j9} to
compute explicit formulas for dual windows.

Condition \eqref{item:ron-shen} also lends itself to proving qualitative
sufficient conditions.  By imposing the diagonal dominance of $R(x)$,
one can derive some conditions on $g$ to guarantee that
$\G (g, \alpha, \beta )$ is a frame. The easiest case is $R(x)$ being
a family of diagonal matrices. In this way one obtains the
% However, there are instances where the Ron-Shen criterion
% is fairly easy to check.  If the window has compact support, the
% Ron-Shen matrices $\RS(x)$ are sparse, which improves inequality
% \eqref{eq:rectangular-lattices-eq1} considerably.
% In the setting of Daubechies, Landau and Landau's
``painless
non-orthogonal expansions'' of Daubechies, Grossman, and
Meyer~\cite{MR836025}. This fundamental result precedes the era of
wavelets and Gabor analysis, and yields all  Gabor frames that are
used for real applications, e.g.,  in signal
analysis or speech processing. %  the Ron-Shen matrix is
% even of diagonal form and yields the following characterization for
% Gabor frames.

\begin{thm}[Painless non-orthogonal expansions]
  \label{sec:criterion-ron-shen-PainLess}
  Suppose $g \in L^{\infty}(\Rd)$ with $\supp g \subseteq [0,L]^d$.
  If $\alpha \leq L$ and $\beta \leq \frac{1}{L}$, then
  $\G(g, \alpha, \beta)$ is a frame if and only if
  \begin{equation*}
    0 < \einf_{x \in \Rd} \sum_{k \in \Z^d} |g(x-\alpha k)|^2 \text{.}
  \end{equation*}
\end{thm}

% We follow the proof given in Gr\"ochenig's survey
% article~\cite{MR3232589}.
% % and uses the Ron-Shen criterion~\eqref{eq:rectangular-lattices-eq1}.

\begin{proof}
  By assumption, we have $\frac{1}{\beta} \geq L$. If $\frac{1}{\beta}
  > L$, then the  supports of $T_{\frac{k}{\beta}}g$ and $T_{\frac{l}{\beta}}g$  are disjoint for
  $k \neq l$; if $\frac{1}{\beta} =  L$, then the supports of
  $T_{\frac{k}{\beta}}g$ and $T_{\frac{l}{\beta}}g$ overlap on a set
  of measure zero and we may modify $g$ so that $T_{\frac{k}{\beta}}g
  \cdot T_{\frac{l}{\beta}}g= 0$ everywhere for $k\neq l$. %  are disjoint for
  % $k \neq l$;   Since the values on a set of measure zero is negligible
  % for $L^2$-functions, we may assume without loss of generality that
  % the supports are disjoint.
  Consequently,
  \begin{equation*}
    \begin{aligned}
      \label{eq:rectangular-lattices-eq5}
      \RS(x)_{k,l} 
      &= \sum_{j \in \Z^d}  g\big(x+j \alpha - \tfrac{k}{\beta}\big)
       \overline{g}\big(x+j \alpha - \tfrac{l}{\beta}\big) \\
      &= \sum_{j \in \Z^d} \big|g\big(x+j \alpha - \tfrac{k}{\beta}\big)\big|^2\,
      \delta_{k,l}
      \text{,}
    \end{aligned}
  \end{equation*}
and thus $R(x)$ is a diagonal matrix for almost all $x$. Clearly,  a
diagonal matrix is bounded and invertible \fif\ its diagonal 
is bounded above and away from zero, therefore the assertion of
Theorem~\ref{sec:criterion-ron-shen-PainLess} follows  immediately.
\end{proof}

Theorem~\ref{sec:rectangular-lattices-Char} can also be reformulated
in terms of frames for $L^2(\td )$. For this we 
recall that the \emph{Zak transform} with respect to the
parameter $\alpha > 0$ is defined by
\begin{equation*}
  \label{eq:zak-transform-ZT}
  Z_{\alpha}f(x, \xi) := \sum_{k \in \Z^d} f(x - \alpha k) e^{2 \pi i \alpha k \cdot \xi}
  \text{.}
\end{equation*}
Most characterizations of a Gabor frame over a rectangular lattice can
be formulated by means of the Zak transform. Here is the general
version attached to Theorem~\ref{sec:rectangular-lattices-Char}.

\begin{thm}
  \label{sec:zak-transform-addChar}
  Let $g \in L^2(\Rd)$ and $\Lambda = \alpha \Z^d \times \beta \Z^d$
  with $\alpha, \beta > 0$. Then the following
  are equivalent:
  \begin{enumerate}[(i)]
  \item $\G(g, \alpha, \beta)$ is a frame for $L^2(\Rd)$.
 
  % %%%%%%%%%%%%%%%%%%%%%%%%%%%%%%%%%%%%%%%%%%%%%%%%%%%%%%%%%%%% 
  % %           COUNTER FOR LATER CHARACTERIZATIONS            
  % \stepcounter{CountRectChar}
  % \setItemnumber{\value{CountRectChar}}
  % %%%%%%%%%%%%%%%%%%%%%%%%%%%%%%%%%%%%%%%%%%%%%%%%%%%%%%%%%%%% 
  \item  $\{ Z_{\frac{1}{\beta}}g(x+\alpha j, \beta \, . \, ): j
  \in \Z^d \}$
  is a frame for $L^2(\T^d)$ for almost all $x \in \Rd$ with frame
  bounds independent of $x$. 
  % \item $\{ \beta^{-d/2}Z_{\frac{1}{\beta}}g\Big(x+\alpha j, \, . \, \Big): j
  % \in \Z^d \}$ is a frame for $L^2(Q_{\beta})$ for almost all $x \in Q_{1/ \beta}$.
  \end{enumerate}
\end{thm}

\begin{proof} 
  By Theorem~\ref{sec:rectangular-lattices-Char}, $\G(g,\alpha,
  \beta)$ is a Gabor frame if and only if there exist positive
  constants $A,B > 0$ such that
  \begin{equation}\label{eq:c3}
    A \|c\|\lnrm^2 \leq \sum_{j \in \Z^d} \Big| \sum_{k \in \Z^d} c_k
    \overline{g}\big(x+\alpha j - \tfrac{k}{\beta}\big)\Big|^2 
    \leq B \|c\|\lnrm^2
  \end{equation}
  for all $c \in \ell^2(\Z^d)$ and almost all $x \in \Rd$. 

  Using Parseval's identity for Fourier series, we interpret the inner
  sum over $k$ as an inner product of periodic $L^2$-functions. The
  Fourier series of $c$ is
  $\hat{c}(\xi ) = \sum _{k\in \zd } c_k e^{2\pi i k \cdot \xi }$, and
  the Fourier series of the sequence
  $\big(g(x+\alpha j - k/\beta) \big)_{k \in \zd }$ (for fixed $x$) is
  precisely the Zak transform
  $$
  Z_{\frac{1}{\beta}}g(x+\alpha j, \beta \xi) = \sum_{k \in \Z^d}
  g\big(x+ \alpha j - \tfrac{k}{\beta}\big) e^{2 \pi i k \cdot \xi}  \,
  .
  $$
  Consequently,
  $$
  \sum_{k \in \Z^d} c_k \overline{g}\big(x+\alpha j -
  \tfrac{k}{\beta}\big) = \int _{\T ^d} \hat{c}(\xi ) \overline{
    Z_{\frac{1}{\beta}}g(x+\alpha j, \beta \xi)} \dif \xi
  $$
  and \eqref{eq:c3} just says that the set
  $\{ Z_{\frac{1}{\beta}}g(x+\alpha j, \beta \cdot ) : j\in \zd \}$ is
  a frame for $L^2(\T ^d)$ for almost all $x\in \rd $. Furthermore, the
  frame bounds can be chosen to be $A$ and $B$ independent of $x$.
\end{proof}

\section{Zak Transform Criteria for Rational Lattices --- The Criteria of Zeevi and Zibulski}
\label{sec:crit-zeevi-zibulski}
All criteria formulated so far are expressed by the invertibility of
an infinite matrix or of an operator on an infinite dimensional
space. For rectangular lattices $\Lambda = \alpha \zd \times \beta \zd
$ with $\alpha \beta \in \Qf$ one may further reduce the effort and
study the invertibility of a family of finite-dimensional matrices. 

% The Ron-Shen criterion of
% Theorem~\ref{sec:rectangular-lattices-Char} states that
% $\G(g, \alpha, \beta)$ is a Gabor frame if and only if the inequality
% \begin{equation}
%   A \|c\|_{\ell^2}^2 \leq \sum_{j \in \Z^d} \Big| \sum_{k \in \Z^d} c_k
%   \overline{g}\Big(x+\alpha j - \frac{k}{\beta}\Big)\Big|^2 
%   \leq B \|c\|_{\ell^2}^2 
% \end{equation}
% is satisfied for all $c \in \ell^2(\Z^d)$ and almost all $x \in \Rd$.
% In Section~\ref{sec:criterion-ron-shen}, we interpreted the sum over
% $k \in \Z^d$ as an inner product of sequences and used Parseval's
% formula to derive a characterization of Gabor frames involving the Zak
% transform.

% With the additional structure of a rational lattice, \ie a rectangular
% lattice where $\alpha \beta \in \Qf$, we may interpret the sum over
% $k \in \Z^d$ as a convolution after a suitable periodization and apply
% Plancherel's theorem on the outer $\ell^2$-norm.  This yields
% characterizations of Zeevi and Zibulski~\cite{MR1448221,MR1601071},
% where only the spectral properties of a family of finite dimensional
% matrices need to be checked.

% Recall that the \index{Zak transform} \emph{Zak transform} for the
% parameter $\alpha > 0$ is defined as
%  \begin{equation}
%   Z_{\alpha}f(x, \xi) = \sum_{k \in \Z^d} f(x - \alpha k)
%   e^{2 \pi i \alpha k \cdot \xi} 
% \end{equation}
% and quasi-periodic on $ \Zdom$ by
% Lemma~\ref{sec:criterion-ron-shen-quasiper}.  Furthermore for
% $f \in L^2(\Rd)$, the Zak transform is defined almost everywhere on
% $\Rtd$ by Lemma~\ref{sec:criterion-ron-shen-ZakBounded}.

Assume that $\alpha \beta = p/q\leq 1$ for
$p,q \in \mathbb{N}$. In order to simplify the labeling of vectors and
matrices, we define $ E_q := \{0,1, \dots , q-1\}^d$ and
$ E_p := \{0,1, \dots, p-1\}^d$. We then write $j\in \zd $ uniquely as
$j=ql+r$ for $l\in \zd$ and $r\in E_q$. Using the quasi-periodicity of
the Zak transform, we obtain
\begin{equation*}
  \label{eq:c5}
  Z_{\frac{1}{\beta}}g(x+\alpha j, \beta \xi) =
  Z_{\frac{1}{\beta}}g\big(x+\tfrac{p}{q\beta } (ql+r), \beta \xi \big) =
  e^{2\pi i pl \cdot \xi}     Z_{\frac{1}{\beta}}g\big(x+
  \tfrac{p}{q\beta}r, \beta \xi \big ) \, .
\end{equation*}
  Thus for rational
values of $\alpha \beta $, we obtain a function system which factors
into certain  complex exponentials and some functions. The frame
property of such a system is characterized in the following lemma. 

\begin{lem}
  \label{l:c1}
  Let $\{h_r : r \in F\} \subseteq L^2(\td )$ be a finite set and
  $p \in \mathbb{N}$ such that $\mathrm{card}\, F \geq \mathrm{card}\,
  E_p = p^d$.  Furthermore, let
  $\A (\xi ) $ be the matrix with entries
  $\A (\xi ) _{r,s} = \overline{h_r} (\xi + \tfrac{s}{p})$ for
  $r\in F, s\in E_p$. Then the following are equivalent:
  \begin{enumerate} [(i)]
  \item The set
    $\{ e^{2\pi i p l \cdot \xi } h_r(\xi) : l\in \zd , r\in F\}$ is a
    frame for $L^2(\td )$.
  \item There exist $A,B>0$ such that the singular values of
    $\A (\xi )$ are contained in $[A^{1/2},B^{1/2}]$ for almost all
    $\xi \in \td $.
  \item There exist $A,B>0$ such that
    $\sigma (\A ^*(\xi ) \A (\xi )) \subseteq [A,B]$ for almost all
    $\xi \in \td $.
  \end{enumerate}
\end{lem}

The condition  $\mathrm{card}\, F \geq p^d $ is essential, otherwise the 
matrix $\A (\xi ) $ cannot be injective and 
$\A ^*(\xi ) \A (\xi )$ cannot be invertible.

\begin{proof}
  For $f\in L^2(\td )$ and $\xi \in Q_{1/p}= [0,\frac{1}{p}]^d$ we
  write  the vector $y(\xi ) = \big(f(\xi +
  \frac{s}{p})\big)_{s\in E_p}$. Then the inner product of $f$ with the
  frame functions $h_r(\xi )  e^{2\pi i p l \cdot \xi }$ can be written as 
  \begin{align*}
    \langle f , h_r e^{2\pi i pl\cdot \xi } \rangle &= \int
    _{[0,\frac{1}{p}]^d} \sum _{s \in E_p} f(\xi +\tfrac{s}{p})
      \overline{h_r} (\xi + \tfrac{s}{p}) e^{-2\pi i pl \cdot \xi} \, \dif \xi \\
   &= \int _{[0,\frac{1}{p}]^d} \big(\A (\xi ) y(\xi )\big)_r  \,
    e^{-2\pi i pl \cdot \xi } \, \dif \xi \, .
  \end{align*}

  Since $\{p^{d/2} e^{2\pi i pl\cdot \xi}: l\in \zd \}$ is an
  orthonormal basis for $L^2(Q_{1/p})$, we now obtain
  \begin{align*}
    \sum _{l\in \zd } \sum _{r\in F} |\langle f, h_r e^{2\pi i pl\cdot
    \xi }\rangle |^2 &=   \sum _{r\in F } \sum _{l\in \zd } \Big| 
                       \int _{Q_{1/p} } (\A (\xi ) y(\xi ))_r e^{-2\pi
                       i pl\cdot \xi } \,  \dif \xi \Big| ^2 \\
                     &= \frac{1}{p^d} \sum _{r\in F}    \int _{Q_{1/p} } |(\A   (\xi ) y(\xi ))_r |^2
                       \, \dif \xi \\
                     &= \frac{1}{p^d}\int _{Q_{1/p}} |\A (\xi ) y(\xi ) |^2 \, \dif \xi \, .
  \end{align*}

  If the singular values of $\A $ are all in an interval
  $[A^{1/2},B^{1/2}]$, then
  $|\A (\xi ) y(\xi ) |^2 = \langle \A (\xi )^* \A (\xi ) y(\xi),
  y(\xi )\rangle \geq A |y(\xi )|^2$.  Therefore

  \begin{align}
    \sum _{l\in \zd } \sum _{r\in F} |\langle f, h_r e^{2\pi i pl\cdot
    \xi }\rangle |^2  &=  \frac{1}{p^d} \int _{Q_{1/p}} |\A (\xi ) y(\xi ) |^2 \,
                        \dif \xi  \notag \\
                      &\geq \frac{A}{p^d}  \int _{Q_{1/p}} | y(\xi ) |^2 \, \dif \xi 
                        \label{eq:c13} \\
                      &= \frac{A}{p^d} \int _{Q_{1/p}} \sum _{s\in E_p} | f(\xi +\tfrac{s}{p} ) |^2 \,
                        \dif \xi  = \frac{A}{p^d} \|f\|_{L^2(\td ) } ^2 \, . \notag 
  \end{align}
  Similarly, for the upper frame inequality. Thus the set
  $\{h_r(\xi ) e^{2\pi i pl\cdot \xi } : l\in \zd, r\in F \}$ is a
  frame for $L^2(\td )$.

  % Conversely, assume that
  % $\{h_r(\xi ) e^{2\pi i pl\cdot \xi } : r\in F, l\in \zd \}$ is a
  % frame for $L^2(\td )$.  Then \eqref{eq:c13} says that for all
  % $f\in L^2(\td )$ with associated vector-valued function
  % $y(\xi ) = \big(f(\xi + \frac{s}{p})\big)_{s\in E_p}$ we must have
  % \begin{equation}
  %   \label{eq:c15}
  %   \frac{1}{p^d} \int _{Q_{1/p}} |\A (\xi ) y(\xi ) |^2 \,
  %   \dif \xi  \geq \frac{A}{p^d} \int _{Q_{1/p}} | y(\xi ) |^2 \, \dif \xi \, .  
  % \end{equation}
  % Given $\xi _0 \in \td $ and $a\in \Cf ^{E_p}$ we choose
  % $$
  %    f (\xi)  = \frac{1}{|B_\delta (0)|^{1/2}} \sum _{s\in E_p} a_s \chi
  %    _{B_\delta (\xi _0)} (\xi - \tfrac{s}{p}) \, ,
  % $$
  % for $\delta >0$ small enough, so that the shifted balls are disjoint
  % in $\td$ and $\int _{Q_{1/p} } |y(\xi )|^2 \, \dif \xi = \|a\|_2^2$. Then
  % $$\A (\xi ) y(\xi ) = \frac{1}{|B_\delta (0)|^{1/2}} \A (\xi ) a \chi
  %    _{B_\delta (\xi _0)} (\xi ) \, ,
  % $$
  % and by \eqref{eq:c15}
  % \begin{align*}
  %    \frac{1}{p^d|B_\delta (0)|}\int _{B_\delta (\xi _0)}  |\A (\xi )
  %   a  |^2   \, \dif \xi &= \frac{1}{p^d} \int _{Q_{1/p}} |\A (\xi ) y(\xi ) |^2 \,
  %                      \dif \xi \\
  %                    & \geq  \frac{A}{p^d} \int _{Q_{1/p}} | y(\xi ) |^2 \, \dif \xi 
  %                      = \frac{A}{p^d} \|a\|_2^2
  %                      \, .  
  % \end{align*}
  % Now let $\delta \to 0$ and apply Lebesgue's differentiation theorem
  % to conclude that $\|\A (\xi _0) a \|_2 \geq A \|a\|_2$ for almost
  % all $\xi_0 \in \td $, as claimed. 
  % *************
  
  Conversely, assume that
  $\{h_r(\xi ) e^{2\pi i pl\cdot \xi } : l\in \zd,  r\in F \}$ is a
  frame for $L^2(\td )$.  Then \eqref{eq:c13} says that for all
  $f\in L^2(\td )$ with associated vector-valued function
  $y(\xi ) = \big(f(\xi + \frac{s}{p})\big)_{s\in E_p}$ we must have
  \begin{equation}
    \label{eq:c17}
    \frac{A}{p^d} \int _{Q_{1/p}} | y(\xi ) |^2 \, \dif \xi \leq
    \frac{1}{p^d} \int _{Q_{1/p}} |\A (\xi ) y(\xi ) |^2 \,
    \dif \xi  \leq \frac{B}{p^d} \int _{Q_{1/p}} | y(\xi ) |^2 \, \dif \xi \, .  
  \end{equation} 
  % \begin{equation}
  %   \frac{A}{p^d} \int _{Q_{1/p}} | y(\xi ) |^2 \, \dif \xi \leq
  %   \frac{1}{p^d} \int _{Q_{1/p}} \langle \A (\xi )^* \A (\xi ) y(\xi),
  %   y(\xi )\rangle\,
  %   \dif \xi  \leq \frac{B}{p^d} \int _{Q_{1/p}} | y(\xi ) |^2 \, \dif \xi \, .  
  % \end{equation}

  % By an elementary argument of measurable choice (see
  % Azoff~\cite{MR0327799}), we may diagonalize $\A (\xi )^* \A (\xi )$
  % with measurable matrix-valued functions $\U, \D$, where each
  % $\U (\xi)$ is a unitary $E_p \times E_p$ matrix and $\D (\xi)$ the
  % diagonal matrix consisting of the eigenvalues of
  % $ \A (\xi )^* \A (\xi )$.

  Now we diagonalize $ \A (\xi )^* \A (\xi )$.  Since $\A^*\A$ is a
  measurable matrix-valued function on $\T^d$, its diagonalization can
  be chosen to be measurable (see Azoff~\cite{MR0327799}).  This means that
  there exist two measurable matrix-valued functions $\U, \D$ such
  that $\U(\xi)$ is a unitary matrix, $\D (\xi)$ is of diagonal form
  and $ \A (\xi )^* \A (\xi ) = \U(\xi)^*\D(\xi)\U(\xi)$ for all
  $\xi \in \T^d$.

% there exists a
%   measurable  diagonalization by an elementary application of
%   measurable choice (see Azoff~\cite{MR0327799}).  This means that
%   there exist two measurable matrix-valued functions $\U, \D$ where
%   each $\U(\xi)$ is a unitary matrix, each $\D (\xi)$ is a diagonal
%   matrix and $ \A (\xi )^* \A (\xi ) = \U(\xi)^*\D(\xi)\U(\xi)$.

  Hence \eqref{eq:c17} is equivalent to
  \begin{equation}
    \label{eq:c19}
    A \int _{Q_{1/p}} | \tilde{y}(\xi ) |^2 \, \dif \xi \leq
    \int _{Q_{1/p}} \langle \D (\xi ) \tilde{y}(\xi ), \tilde{y}(\xi) \rangle \,
    \dif \xi  \leq B \int _{Q_{1/p}} | \tilde{y}(\xi ) |^2 \, \dif \xi 
  \end{equation}
  for all vector-valued functions $\tilde{y}(\xi) = \U(\xi) y(\xi)$
  with components in $L^2(Q_{1/p})$.  

  Clearly, inequality \eqref{eq:c19} can only hold if
  $\sigma(\D(\xi)) = \sigma ( \A (\xi )^* \A (\xi )) \subseteq [A,B]$
  for almost all $\xi \in Q_{1/p}$.
\end{proof}

We now apply this lemma to the set
$\{e^{2\pi i p l \cdot \xi } Z_{\frac{1}{\beta} } (x+\tfrac{p}{q\beta}
r,\beta \xi ) : l\in \Z^d, r\in E_q \}$
and obtain the characterization of Zeevi and Zibulski for rational
rectangular lattices~\cite{MR1448221,MR1601071}.

\begin{thm}
  Let $g\in \lrd $ and $\alpha \beta = p/q \in \mathbb{Q}$ with
  $p/q \leq 1$.  For $x,\xi \in \rd $ let $\Q (x,\xi )$ be the matrix
  with entries
  \begin{equation*}
    \label{eq:c6}
    \Q (x,\xi )_{r,s} = Z_{\frac{1}{\beta} } g\big(x+ \tfrac{p}{\beta q} r, \beta \xi +
    \tfrac{\beta s}{p}\big) \qquad \forall r\in E_q, s\in E_p \, .
  \end{equation*}
  The Gabor family $\G (g,\alpha , \beta  )$ is a frame
  for $\lrd $ \fif\ the singular values of $\Q (x,\xi )$ are
  contained in an interval $[A^{1/2},B^{1/2}] \subseteq (0,\infty )$ for almost
  all $x, \xi \in \rd $.
\end{thm}

% %Note that $|E_q \geq |E_p|$ is always satisfied as $p/q \leq 1$.

% Check: should be $Z.. (..., \beta \xi )$??

\begin{proof}
  By Theorem~\ref{sec:zak-transform-addChar}, $\G (g,\alpha, \beta ) $
  is a frame \fif\
  $\{ Z_{\frac{1}{\beta}}g(x+\alpha j, \beta \xi ): j \in \Z^d \} = \{
  e^{2\pi i p l \cdot \xi } Z_{\frac{1}{\beta}}g(x+\alpha r, \beta \xi
  ): l\in \zd , r \in E_q\} $
  is a frame for $L^2(\T^d)$.  Now, the claim follows from
  Lemma~\ref{l:c1} with the functions
  $h_r(\xi ) = Z_{\frac{1}{\beta}}g(x+\alpha r, \beta \xi )$.
\end{proof}

The Zak transform has been used frequently to derive theoretical
properties of Gabor frames.  The Zeevi-Zibulski matrices in particular
are very useful for computational issues, and several important
counter-examples have been discovered first through numerical tests
before being proved rigorously~\cite{MR3572909,MR3027914}. On the other hand, it
seems to be very difficult to apply directly and decide rigorously
whether a concrete Gabor system is a frame or not.

\section{Further Characterizations}

So far we have discussed characterization of Gabor frames that work
for arbitrary windows in $\lrd $. On a technical level, we have not
used more than the Poisson summation formula. Under mild additional
conditions that are standard in time-frequency analysis, one can prove
further characterizations for Gabor frames. These, however, require
additional and  more
advanced mathematical tools, such as spectral invariance, a
non-commutative version of Wiener's lemma or  Beurling's
method of weak limit. For this reason, we state  these  characterizations without proofs.

\subsection{The Wiener amalgam space and irrational lattices}
This condition refines Theorem~\ref{sec:rectangular-lattices-Char} for
irrational lattices. As the appropriate class of window we use the
Wiener amalgam space $W_0 = W(C,\ell ^1)$. It consists of all
continuous functions $g$ for which the norm
$$
\|g\|_W = \sum  _{k\in \zd } \sup _{x\in Q_1} |g(x+k)|
$$
is finite. 

\begin{thm}[~\cite{MR1921985}]
  Assume that $g\in W_0$ %  is continuous and satisfies the condition $\sum
  % _{k\in \zd } \sup _{x\in Q_1} |g(x+k)| = \|g\|_W <\infty $,
  and $\Lambda = \alpha \Z^d \times \beta \Z^d$ with
  $\alpha \beta \notin \Qf$.  Then $\G(g, \alpha, \beta)$ is a frame
  for $L^2(\Rd)$ \fif\ there exists some $x_0 \in Q_{\alpha}$ such
  that $\RS(x_0)$ is invertible on $\ell^2(\Z^d)$.
\end{thm}

Thus for irrational
lattices it suffices to check the invertibility of a single Ron-Shen
matrix $R(x)$ instead of all matrices. Although this condition looks
useful, it has not yet found any applications.

\subsection{Janssen's criterium without inequalities}

\begin{thm}[\cite{grrost17}] \label{thtp}
  Assume that $g\in W_0$ and $\Lambda = \alpha \zd \times \beta \zd
  $. Then $\G (g, \alpha  ,\beta ) $ is a frame if and only if the
  pre-Gramian $P(x)$ is one-to-one on $\ell ^\infty (\zd )$ for all
  $x\in \rd $. 
\end{thm}
Put differently,  to show that $\G (g,\alpha , \beta ) $ is a frame,
one has to show that
\begin{equation*}
  \sum _{k\in \zd } c_k g(x+\alpha j - \tfrac{k}{\beta}) = 0  \quad
  \Longrightarrow \quad  c \equiv 0 \, ,
\end{equation*}
with the added subtlety that $c$ is only a bounded sequence, but not necessarily
in $\ell ^2(\zd )$, as  is the case in 
Theorem~\ref{sec:rectangular-lattices-Char}.

In general it is easier
to verify the injectivity of  an operator than to prove its
invertibility, therefore Theorem~\ref{thtp} is a strong result. It has
been applied successfully for the study of totally positive windows of
Gaussian type in ~\cite{grrost17} and carries potential for further applications.

\subsection{Gabor frames without inequalities}
This group of conditions holds for arbitrary lattices and windows in
$M^1(\rd )$. As is well-known, the modulation space $M^1(\Rd)$ is a
natural condition in many problems in time-frequency analysis, because
it is invariant under the Fourier transform and many other
transformations. By choosing a suitable norm, $M^1(\Rd)$ becomes a
Banach space and its dual space $M^\infty(\Rd)$ consists of all
tempered distributions $f\in \Sc ' (\rd )$ that satisfy
$$
\sup _{z\in \rdd } |V_\varphi f(z) | < \infty $$
for some (or equivalently, for all) Schwartz functions $\phi $.

\begin{thm}[~\cite{MR2380004}]
  \label{sec:char-with-ineq-BigThm}
  Let $g \in M^1(\Rd)$ and $\Lambda \subseteq \Rtd$ be a lattice.
  Then the following are equivalent:
  \begin{enumerate}[(i)]
  \item $\G(g,\Lambda)$ is a frame for $L^2(\Rd)$, \ie $S_{g,\Lambda}$
    is invertible on $L^2(\Rd)$.
  \item The frame operator $S_{g,\Lambda}$ is one-to-one on $M^\infty(\Rd)$.
  \item The analysis operator $C_{g,\Lambda} : f \mapsto \big( \langle f,
    \pi (\lambda ) g \rangle \big) _{\lambda \in \Lambda }$  is one-to-one
    from $M^\infty(\Rd)$ to
    $\ell^\infty(\Lambda)$.
  \item The synthesis operator
    $D_{g,\Lambda^\circ} : c \mapsto \sum _{\lambda \in \Lambda ^\circ
    } c_\lambda \pi (\lambda )g$
    is one-to-one from $\ell^\infty(\Lambda^\circ)$ to
    $M^\infty(\Rd)$. \label{item: D injective}
  \item The Gramian operator $G_{g,\Lambda^\circ}$ is one-to-one on $\ell^\infty(\Lambda^\circ)$.
  \end{enumerate}
\end{thm}
Conceptually, it seems easier to verify that an operator is one-to-one,
therefore one may hope that these conditions will become useful when
research on Gabor frames will move from rectangular lattices towards
arbitrary ones.

% %\nocite{*}
% \bibliographystyle{abbrv}
% \bibliography{gaborchar}

\begin{thebibliography}{10}

\bibitem{MR3192621}
G.~Ascensi, H.~G. Feichtinger, and N.~Kaiblinger.
\newblock {Dilation of the {W}eyl symbol and {B}alian-{L}ow theorem}.
\newblock {\em Trans. Amer. Math. Soc.}, 366(7):3865--3880, 2014.

\bibitem{MR0327799}
E.~A. Azoff.
\newblock {Borel measurability in linear algebra}.
\newblock {\em Proc. Amer. Math. Soc.}, 42:346--350, 1974.

\bibitem{bekka04}
B.~Bekka.
\newblock {Square integrable representations, von {N}eumann algebras and an
  application to {G}abor analysis}.
\newblock {\em J. Fourier Anal. Appl.}, 10(4):325--349, 2004.

\bibitem{BHW95}
J.~J. Benedetto, C.~Heil, and D.~F. Walnut.
\newblock {Differentiation and the {B}alian--{L}ow theorem}.
\newblock {\em J. Fourier Anal. Appl.}, 1(4):355--402, 1995.

\bibitem{MR3495345}
O.~Christensen.
\newblock {\em {An introduction to frames and {R}iesz bases}}.
\newblock {Applied and Numerical Harmonic Analysis}. Birkh{\"a}user/Springer,
  [Cham], second edition, 2016.

\bibitem{MR2563261}
O.~Christensen, H.~O. Kim, and R.~Y. Kim.
\newblock {Gabor windows supported on {$[-1,1]$} and compactly supported dual
  windows}.
\newblock {\em Appl. Comput. Harmon. Anal.}, 28(1):89--103, 2010.

\bibitem{CP06}
W.~Czaja and A.~M. Powell.
\newblock {Recent developments in the {B}alian-{L}ow theorem}.
\newblock In {\em {Harmonic analysis and applications}}, {Appl. Numer. Harmon.
  Anal.}, pages 79--100. Birkh{\"a}user Boston, Boston, MA, 2006.

\bibitem{daubechies92}
I.~Daubechies.
\newblock {\em {Ten lectures on wavelets}}.
\newblock Society for Industrial and Applied Mathematics (SIAM), Philadelphia,
  PA, 1992.

\bibitem{MR836025}
I.~Daubechies, A.~Grossmann, and Y.~Meyer.
\newblock {Painless nonorthogonal expansions}.
\newblock {\em J. Math. Phys.}, 27(5):1271--1283, 1986.

\bibitem{MR1350701}
I.~Daubechies, H.~J. Landau, and Z.~Landau.
\newblock {Gabor time-frequency lattices and the {W}exler-{R}az identity}.
\newblock {\em J. Fourier Anal. Appl.}, 1(4):437--478, 1995.

\bibitem{DS52}
R.~J. Duffin and A.~C. Schaeffer.
\newblock {A class of nonharmonic {F}ourier series}.
\newblock {\em Trans. Amer. Math. Soc.}, 72:341--366, 1952.

\bibitem{MR1021139}
H.~G. Feichtinger and K.~Gr{\"o}chenig.
\newblock {Banach spaces related to integrable group representations and their
  atomic decompositions. {I}}.
\newblock {\em J. Funct. Anal.}, 86(2):307--340, 1989.

\bibitem{fg92chui}
H.~G. Feichtinger and K.~Gr{\"o}chenig.
\newblock {Gabor Wavelets and the {H}eisenberg group: {G}abor expansions and
  short time Fourier transform from the group theoretical point of view}.
\newblock In C.~K. Chui, editor, {\em {Wavelets: A tutorial in theory and
  applications}}, pages 359--398. Academic Press, Boston, MA, 1992.

\bibitem{FK04}
H.~G. Feichtinger and N.~Kaiblinger.
\newblock {Varying the time-frequency lattice of {G}abor frames}.
\newblock {\em Trans. Amer. Math. Soc.}, 356(5):2001--2023 (electronic), 2004.

\bibitem{MR1601091}
H.~G. Feichtinger and W.~Kozek.
\newblock {Quantization of {TF} lattice-invariant operators on elementary {LCA}
  groups}.
\newblock In {\em {Gabor analysis and algorithms}}, {Appl. Numer. Harmon.
  Anal.}, pages 233--266. Birkh{\"a}user Boston, Boston, MA, 1998.

\bibitem{MR2264211}
H.~G. Feichtinger and F.~Luef.
\newblock {Wiener amalgam spaces for the fundamental identity of {G}abor
  analysis}.
\newblock {\em Collect. Math.}, (Vol. Extra):233--253, 2006.

\bibitem{ga46}
D.~{G}abor.
\newblock {{T}heory of communication}.
\newblock {\em {J}. {I}{E}{E}}, 93(26):429--457, 1946.

\bibitem{MR1843717}
K.~Gr{\"o}chenig.
\newblock {\em {Foundations of time-frequency analysis}}.
\newblock {Applied and Numerical Harmonic Analysis}. Birkh{\"a}user Boston,
  Inc., Boston, MA, 2001.

\bibitem{MR2380004}
K.~Gr{\"o}chenig.
\newblock {Gabor frames without inequalities}.
\newblock {\em Int. Math. Res. Not. IMRN}, (23):Art. ID rnm111, 21, 2007.

\bibitem{MR3232589}
K.~Gr{\"o}chenig.
\newblock {The mystery of {G}abor frames}.
\newblock {\em J. Fourier Anal. Appl.}, 20(4):865--895, 2014.

\bibitem{MR1921985}
K.~Gr{\"o}chenig and A.~J. E.~M. Janssen.
\newblock {Letter to the editor: a new criterion for {G}abor frames}.
\newblock {\em J. Fourier Anal. Appl.}, 8(5):507--512, 2002.

\bibitem{MR3336091}
K.~Gr{\"o}chenig, J.~Ortega-Cerd{\`a}, and J.~L. Romero.
\newblock {Deformation of {G}abor systems}.
\newblock {\em Adv. Math.}, 277:388--425, 2015.

\bibitem{grrost17}
K.~Gr{\"o}chenig, J.~L. Romero, and J.~St{\"o}ckler.
\newblock {Sampling theorems for shift-invariant spaces, {G}abor frames, and
  totally positive functions}.
\newblock {\em Invent. Math.}, 211(3):1119--1148, 2018.

\bibitem{MR3053565}
K.~Gr{\"o}chenig and J.~St{\"o}ckler.
\newblock {Gabor frames and totally positive functions}.
\newblock {\em Duke Math. J.}, 162(6):1003--1031, 2013.

\bibitem{MR2313431}
C.~Heil.
\newblock {History and evolution of the density theorem for {G}abor frames}.
\newblock {\em J. Fourier Anal. Appl.}, 13(2):113--166, 2007.

\bibitem{MR3419761}
M.~S. Jakobsen and J.~Lemvig.
\newblock {Density and duality theorems for regular {G}abor frames}.
\newblock {\em J. Funct. Anal.}, 270(1):229--263, 2016.

\bibitem{MR1310658}
A.~J. E.~M. Janssen.
\newblock {Signal analytic proofs of two basic results on lattice expansions}.
\newblock {\em Appl. Comput. Harmon. Anal.}, 1(4):350--354, 1994.

\bibitem{MR1350700}
A.~J. E.~M. Janssen.
\newblock {Duality and biorthogonality for {W}eyl-{H}eisenberg frames}.
\newblock {\em J. Fourier Anal. Appl.}, 1(4):403--436, 1995.

\bibitem{MR1601115}
A.~J. E.~M. Janssen.
\newblock {The duality condition for {W}eyl-{H}eisenberg frames}.
\newblock In {\em {Gabor analysis and algorithms}}, {Appl. Numer. Harmon.
  Anal.}, pages 33--84. Birkh{\"a}user Boston, Boston, MA, 1998.

\bibitem{MR1964306}
A.~J. E.~M. Janssen.
\newblock {On generating tight {G}abor frames at critical density}.
\newblock {\em J. Fourier Anal. Appl.}, 9(2):175--214, 2003.

\bibitem{MR1955931}
A.~J. E.~M. Janssen.
\newblock {Zak transforms with few zeros and the tie}.
\newblock In {\em {Advances in {G}abor analysis}}, {Appl. Numer. Harmon.
  Anal.}, pages 31--70. Birkh{\"a}user Boston, Boston, MA, 2003.

\bibitem{MR3572909}
J.~Lemvig and K.~{Haahr Nielsen}.
\newblock {Counterexamples to the {B}-spline conjecture for {G}abor frames}.
\newblock {\em J. Fourier Anal. Appl.}, 22(6):1440--1451, 2016.

\bibitem{MR3027914}
Y.~Lyubarskii and P.~G. Nes.
\newblock {Gabor frames with rational density}.
\newblock {\em Appl. Comput. Harmon. Anal.}, 34(3):488--494, 2013.

\bibitem{MR1188007}
Y.~I. Lyubarski\u\i.
\newblock {Frames in the {B}argmann space of entire functions}.
\newblock In {\em {Entire and subharmonic functions}}, volume~11 of {\em {Adv.
  Soviet Math.}}, pages 167--180. Amer. Math. Soc., Providence, RI, 1992.

\bibitem{neumann}
J.~v. Neumann.
\newblock {\em {Mathematische {G}rundlagen der {Q}uantenmechanik}}.
\newblock Springer, Berlin, 1932.
\newblock English translation: ``Mathematical foundations of quantum
  mechanics,'' Princeton Univ. Press, 1955.

\bibitem{rieffel88}
M.~A. Rieffel.
\newblock {Projective modules over higher-dimensional noncommutative tori}.
\newblock {\em Canad. J. Math.}, 40(2):257--338, 1988.

\bibitem{MR1460623}
A.~Ron and Z.~Shen.
\newblock {Weyl-{H}eisenberg frames and {R}iesz bases in
  {$L_2(\mathbf{R}^d)$}}.
\newblock {\em Duke Math. J.}, 89(2):237--282, 1997.

\bibitem{MR1173117}
K.~Seip.
\newblock {Density theorems for sampling and interpolation in the
  {B}argmann-{F}ock space. {I}}.
\newblock {\em J. Reine Angew. Math.}, 429:91--106, 1992.

\bibitem{MR0304972}
E.~M. Stein and G.~Weiss.
\newblock {\em {Introduction to {F}ourier analysis on {E}uclidean spaces}}.
\newblock Princeton University Press, Princeton, N.J., 1971.
\newblock Princeton Mathematical Series, No. 32.

\bibitem{walnut93}
D.~F. Walnut.
\newblock {Lattice size estimates for {G}abor decompositions}.
\newblock {\em Monatsh. Math.}, 115(3):245--256, 1993.

\bibitem{Wexler:1990:DGE:104676.104677}
J.~Wexler and S.~Raz.
\newblock {Discrete Gabor Expansions}.
\newblock {\em Signal Process.}, 21(3):207--220, Oct. 1990.

\bibitem{MR1601071}
Y.~Y. Zeevi, M.~Zibulski, and M.~Porat.
\newblock {Multi-window {G}abor schemes in signal and image representations}.
\newblock In {\em {Gabor analysis and algorithms}}, {Appl. Numer. Harmon.
  Anal.}, pages 381--407. Birkh{\"a}user Boston, Boston, MA, 1998.

\bibitem{MR1448221}
M.~Zibulski and Y.~Y. Zeevi.
\newblock {Analysis of multiwindow {G}abor-type schemes by frame methods}.
\newblock {\em Appl. Comput. Harmon. Anal.}, 4(2):188--221, 1997.

\end{thebibliography}

\end{document}